\documentclass[a4paper,10pt]{amsart}
\usepackage{a4wide}
\usepackage{amsmath,amsfonts,amssymb,amsthm}
\usepackage{tikz}
\usetikzlibrary{calc,decorations.pathreplacing}
\usepackage{latexsym,graphicx,color,url,enumerate}
\usepackage[shortlabels]{enumitem} %pour commencer les enumerations a des nombres differents

\theoremstyle{plain}
\newtheorem{theorem}{Theorem}[section]
\newtheorem{thmintro}{Theorem}[section]
\newtheorem*{theorem*}{Theorem}
\newtheorem*{conjecture*}{Conjecture}
\newtheorem{lemma}[theorem]{Lemma}
\newtheorem{corollary}[theorem]{Corollary}
\newtheorem{proposition}[theorem]{Proposition}

\newtheorem{questionintro}[thmintro]{Question}

\theoremstyle{definition}
\newtheorem{definition}[theorem]{Definition}
\newtheorem{remark}[theorem]{Remark}
\newtheorem{example}[theorem]{Example}

\definecolor{darkblue}{rgb}{0,0,0.7} % darkblue color
\newcommand{\darkblue}{\color{darkblue}} % darkblue command
\newcommand{\defn}[1]{\emph{\darkblue #1}} % emphasis of a definition

\newcommand{\ball}{\mathcal{B}}
\newcommand{\sphere}{\mathcal{S}}
\newcommand{\maxlength}{\omega}
\newcommand{\ssm}{\smallsetminus}
\newcommand{\eqdef}{:=}
\newcommand{\N}{\mathbb{N}}
\newcommand{\Z}{\mathbb{Z}}
\newcommand{\join}{\ast}
\DeclareMathOperator{\length}{length}
\DeclareMathOperator{\ops}{ops} % one-point-suspension
\DeclareMathOperator{\vdist}{vdist} % vertex distance
\DeclareMathOperator{\lk}{lk} % link
\DeclareMathOperator{\st}{st} % star
\DeclareMathOperator{\del}{del} % deletion
\DeclareMathOperator{\stsubd}{stell} % stellar subdivision

\title{Hirsch polytopes with exponentially long combinatorial~segments}

\author[J.-P.~Labb\'e]{Jean-Philippe Labb\'e$^{ \clubsuit}$}
\address[J.-P. Labb\'e]{Einstein Institute of Mathematics, Hebrew University of Jerusalem, Jerusalem 91904, Israel}
\email{labbe@math.fu-berlin.de}
\urladdr{http://www.math.huji.ac.il/~labbe}
\thanks{$^{ \clubsuit}$With the support of a FQRNT post-doctoral fellowship and a post-doctoral ISF grant (805/11)}

\author[T.~Manneville]{Thibault Manneville$^{ \diamondsuit}$}
\address[T.~Manneville]{LIX-UMR 7161 {\'E}cole Polytechnique, 91128 Palaiseau Cedex, France}
\email{thibault.manneville@lix.polytechnique.fr}
\urladdr{http://www.lix.polytechnique.fr/~manneville/}
\thanks{$^{ \diamondsuit}$With the support of an {\'E}cole Polytechnique Gaspard Monge doctoral grant and with partial support of French ANR grant EGOS (12 JS02 002 01)}

\author[F.~Santos]{Francisco Santos$^{ \heartsuit}$}
\address[F.~Santos]{Departamento de Matem\'aticas, Estad\'{\i}stica y Computaci\'on, Universidad de Cantabria, 39012 Santander, Spain}
\email{francisco.santos@unican.es}
\urladdr{http://personales.unican.es/santosf}
\thanks{$^{ \heartsuit}$With the support of the Spanish Ministry of Science through grants MTM2011-22792 and MTM2014-54207-P}

\keywords{Banner complex \and Flag complex \and Diameter \and Normal complex \and Simplicial complex}
\subjclass[2010]{Primary 52B05; Secondary 90C60}

\begin{document}

\begin{abstract}
In their paper proving the Hirsch bound for flag normal simplicial complexes (Math. Oper.~Res.~2014)
Adiprasito and Benedetti define the notion of~\emph{combinatorial segment}.
The study of the maximal length of these objects provides the upper bound~$O(n2^d)$ for the diameter of any normal pure simplicial complex of dimension~$d$ with~$n$ vertices, and the Hirsch bound $n-d$ if the complexes are, moreover, flag.
In the present article, we propose a formulation of combinatorial segments which is equivalent but more local, by introducing the notions of monotonicity and conservativeness of dual paths in pure simplicial complexes.
We use this definition to investigate further properties of combinatorial segments.
Besides recovering the two stated bounds, we show a refined bound for banner complexes, and study the behavior of the maximal length of combinatorial segments with respect to two usual operations, namely join and one-point suspension.
Finally, we show the limitations of combinatorial segments by constructing pure normal simplicial complexes in which all combinatorial segments between two particular facets
achieve the length $\Omega(n2^{d})$. This includes vertex-decomposable---therefore Hirsch---polytopes.
\end{abstract}

\maketitle

%%%%%%%%%%%%%%%%%%%%%%%%%%%%%%
%%%%%%% Introduction %%%%%%%%%
%%%%%%%%%%%%%%%%%%%%%%%%%%%%%%

\section{Introduction}
One of the main open questions in polyhedral combinatorics is the so-called polynomial Hirsch Conjecture:
\begin{quote}
	\emph{Is there an upper bound for the (combinatorial) diameter of the graph of every $d$-polyhedron with $n$ facets that is polynomial in $n$ and $d$?}
\end{quote}
Remember that it was conjectured by Hirsch in 1957 that $n-d$ is a valid upper bound. We say a polytope or polyhedron is \emph{Hirsch} if it satisfies this bound. 
Although the Hirsch Conjecture was disproved by Klee and Walkup in the general case~\cite{KleeWalkup} and by Santos in the bounded case~\cite{Santos-hirsch}, the known counterexamples exceed the bound only by a small fraction (25\% in the unbounded case, 5\% in the bounded case, see \cite{Santos-hirsch,MatschkeSantosWeibel}).
In contrast, the best upper bounds known are not polynomial. If we denote by $H(d,n)$ the maximum diameter of $d$-dimensional polytopes with $n$ facets, these bounds are:
\begin{equation}
\label{eqn:two_bounds}
H(d,n) \le \frac{2}{3}2^{d-3} n,
\qquad
H(d,n) \le (n-d)^{\log_2  d }.
\end{equation}

The first bound is \emph{linear in fixed dimension}. Except for a constant factor, which was an improvement by Barnette~\cite{barnette_upperbound_1974}, this bound was first proved by Larman~\cite{larman_paths_1970}. The second, \emph{quasi-polynomial} bound, was first proved by Kalai and Kleitman~\cite{kalai_quasipolynomial_1992}, although the version we state is an improvement by Todd~\cite{todd_improved_KK}.

An approach to the question that has been attempted since long, starting with the introduction of \defn{abstract polytopes} by Adler and Dantzig~\cite{AdlerDantzig}, is to generalize it to the setting of simplicial complexes. It is known that $H(d,n)$ is attained at a simple polytope, which is topologically dual to a simplicial sphere of dimension $d-1$. Thus, $H(d,n)$ is bounded by the maximum diameter of dual graphs of all simplicial $(d-1)$-spheres with $n$ vertices, and generalizing the question to arbitrary pure simplicial complexes may help in finding the right proof strategy. 
Here and in the rest of the paper, a \emph{pure simplicial complex} of dimension $d-1$ is a subset~$C$ of $\binom{[n]}{d}$.
The elements of $C$ are called \emph{facets} and two facets are adjacent if they differ in a single element. 
This defines a \emph{facet-adjacency} graph (or \emph{dual graph}) of $C$. 
The (dual) diameter of $C$ is the combinatorial diameter of this graph. 

However, if one generalizes the question too much then exponential diameters arise:

\begin{thmintro}[Santos~\cite{Santos-progress}]
The maximum diameter of pure simplicial $(d-1)$-complexes with $n$ vertices grows as $n^{\Theta(d)}$.
\end{thmintro}

Hence, some condition is needed if one wants to have (hopes of) a polynomial upper bound. The standard consensus is to look only at \emph{normal} complexes. That is, complexes in which for every $\sigma\in [n]$ the following subcomplex (called the \emph{star} of $\sigma$) has a connected dual graph:
\[
\st_C(\sigma) := \{ F\in C : \sigma\subset F\}.
\]
Equivalently, $C$ is normal if the \emph{link} of every face of codimension at least two is connected~\cite[Lemma 3.1.2]{IzmestievJoswig}, or if $C$ can be obtained from a family of disjoint $d$-simplices by gluing only along codimension-one faces of them.
The adjective ``normal'' for these complexes is quite established if they are pseudomanifolds,
in which case the condition is also equivalent to  the $n$-th local homology group $H_n(|C|, |C|\ssm \{p\})$ being isomorphic to $\Z$ for every point $p$ (see e.g.~\cite[Sect.~4]{GoreskyMacpherson}). 
Without the pseudomanifold condition, the name has been used in~\cite{AdiprasitoBenedetti, Santos-progress}. In the context of diameter bounds, pure simplicial complexes with this property have been called \emph{locally connected}~\cite{AdlerDantzig} and \emph{ultraconnected}~\cite{Kalai}. Adler and Dantzig~\cite{AdlerDantzig} call normal pseudomanifolds without boundary \emph{abstract polytopes}.
The property arises also in the context of unfoldings of simplicial complexes~\cite{Mohar,IzmestievJoswig} under the name \emph{locally strongly connected}.

Evidence that normality is the right condition to consider is that the two bounds stated in~\eqref{eqn:two_bounds} for $H(d,n)$ are still valid, with minor adjustments, for the diameters of all normal complexes. This was already known to Kalai and Kleitman~\cite{kalai_quasipolynomial_1992}, and it was recently highlighted by the work of Eisenbrand et al.~\cite{eisenbrand_diameter_2010}, who considered an even more abstract setting and proved the two bounds in this setting. Their work, in fact, suggests the following conjecture, much stronger than the polynomial Hirsch Conjecture:

\begin{conjecture*}[{H\"ahnle \cite[Conjecture~7.0.5]{hahnle_diplom_2008}}]
\label{conj:haehnle}
The diameter of every normal $(d-1)$-complex with $n$ vertices is at most $(n-1)d$.
\end{conjecture*}

Normality also plays a role in Adiprasito and Benedetti's proof of the (original) Hirsch bound for \emph{flag polytopes}~\cite{AdiprasitoBenedetti}; indeed, what they prove is the Hirsch bound for the diameter of every flag, normal simplicial complex. 
In this paper we revisit one of the two proofs of this result given by Adiprasito and Benedetti and study what it implies for normal complexes that are not flag. 

Let us be more precise.  Adiprasito--Benedetti introduce paths of certain types in dual graphs of pure complexes, that they call \emph{combinatorial segments}. They then prove that:
\begin{itemize}
\item Between every pair of facets in a normal (not necessarily flag) complex there is always a combinatorial segment (see Corollary~\ref{coro:path-existence-2}).
\item Combinatorial segments in flag normal complexes are nonrevisiting, hence they satisfy the Hirsch bound (see Section~\ref{subsection:nonrevisiting}).
\end{itemize}
Hereafter, we recast these proofs and the concept of combinatorial segment.

In Section~\ref{sec:combinatorialsegments} we introduce the notion of \emph{monotone conservative} paths by extracting the key properties of Adiprasito and Benedetti's combinatorial segments.
Although the two concepts are equivalent (see Theorem~\ref{thm:equivalenceDefinitions}) our formulation is local: a path is \emph{monotone and conservative} if each step from a facet $F_i$ to the next facet $F_{i+1}$ satisfies certain conditions. The original definition goes by posing global conditions on the path via a double recursion: on the vertex-distance between $F_1$ and $F_2$ and on the dimension of $C$. Here, we call \emph{vertex-distance} between two subsets of vertices of $C$ the distance between them in the $1$-skeleton  (the``usual'' graph, whose vertices and edges are the $0$-faces and $1$-faces) of $C$.
We achieve this local, yet equivalent, definition by considering the vertex set of each facet along the path to be ordered, so that the ordering encodes the recursions.

More precisely, we say that the order on~$F_1=\{x_1,\dots,x_d\}$ is~\emph{admissible} with respect to~a target set $S$ if the first vertex~$x_1$ in this order realizes the vertex-distance between~$F_1$ and~$S$, and if the rest of $F_1$, which is itself a facet in the link of~$x_1$ in the complex, is admissibly ordered with respect to a certain set $S'$ of neighbors of $x_1$, derived from $S$. 
We are interested primarily in the case where $S$ is the target facet $F_2$ of our path, but the lack of control on $S'$ forces us to consider more general cases for $S$.
We also define the~\emph{vector of distances} from $F_1$ to $S$, whose coordinates correspond to the vertices of~$F_1$ in their given order.
Its first coordinate is the vertex-distance of~$x_1$ to~$S$ in $C$ and the rest is the vector of distances from $F_1\ssm\{x_1\}$ to $S'$ in the link of~$x_1$.

Now for a step~$\Gamma=[F_1,F_2]$ in a dual path in~$C$, consisting of two facets admissibly ordered with respect to some set~$S$, we concentrate on the first vertices of their orders. One would expect that~$F_2$ will be ``closer'' from~$S$ than~$F_1$ if its first vertex is at smaller vertex-distance from~$S$ than the first vertex of~$F_1$ and that this goes through via the recursion.
This suggest to focus on paths whose steps are all~\emph{monotone}, i.~e., such that the vector of distances lexicographically decreases at each step.
We also require some additional ``control'' property that we call~\emph{conservativeness} and which can informally be expressed as: the step from $F_1$ to $F_2$ is conservative if $F_2$ tries to keep the order established in $F_1$ as much as admissibility allows.

In Section~\ref{sec:upperbounds} we rephrase the proof of the Hirsch bound for flag normal complexes in the language of monotone conservative paths. 
We also include a proof of the Larman bound for normal complexes in this language, and we generalize the bound to the so-called banner complexes:

\begin{thmintro}[Adiprasito--Benedetti~\cite{AdiprasitoBenedetti}, see Theorem~\ref{thm:BarnetteLarman} and Corollary~\ref{coro:hirsch-for-flag}]
\label{thm:BarnetteLarman-intro}
The length of any monotone conservative path in a~$(d-1)$-dimensional pure normal complex~$C$ on~$n$ vertices ($d\ge2$) is at most~$n2^{d-2}$.
Moreover, if~$C$ is a pseudomanifold, the same statement holds with the bound~$n2^{d-3}$; and if~$C$ is flag, the same statement holds with the Hirsch bound~$n-d$.
\end{thmintro}

\begin{thmintro}[Novik 2013, unpublished, see Theorem~\ref{thm:upperBound}]
\label{thm:upperBound-intro}
Let $k\ge2$.
If~$C$ is a $k$-banner  pure normal complex on~$n$ vertices, then between every distinct facets $F_1$ and $F_2$ of $C$, there is a monotone conservative path of length bounded by~$n2^{k-2}$.
\end{thmintro}

The property of being \emph{$k$-banner} is a generalization of flagness (which is itself equivalent to $2$-banner). See the precise definitions in Section~\ref{subsec:BannerComplexes}.  Theorem~\ref{thm:upperBound-intro} interpolates (modulo a factor of $2$ and an additive term of $d$, respectively) between two of the bounds of Theorem~\ref{thm:BarnetteLarman-intro}: Since every $(d-1)$-complex is $(d+1)$-banner, we retrieve the bound $2^{d-1}n$ for arbitrary normal complexes. On the other extreme, substituting $k$ by $2$ into the theorem gives a bound of $n$ for flag normal complexes.

Theorem~\ref{thm:upperBound-intro} was proven by Novik (unpublished, 2013) as an attempt to answer the following stronger question:

\begin{questionintro}
Let~$k\ge2$ and~$C$ be a pure complex on~$n$ vertices in which any minimal nonface has dimension at most~$k$, then the diameter of~$C$ is bounded by~$n2^{k-2}$.
\end{questionintro}

Finally in Section~\ref{sec:lowerbound} we show the limitations of combinatorial segments by constructing monotone conservative paths of exponential length in normal complexes. 
We do this in two versions. In the first one the complex is a topological ball, in the second it is the boundary complex of a vertex-decomposable simplicial polytope:

\begin{thmintro}[See Theorem~\ref{thm:exampleSphere}]
\label{thm:exampleSphere-intro}
For every $d\ge 2$ and every $N\ge 4$ there is a simplicial $d$-polytope with $N+\Theta(d^2)$ vertices 
 with vertex-decomposable boundary complex and two facets $F_1$ and $F_2$ in it such that every  monotone conservative path between them has length at least $2^{d-3} N$.
\end{thmintro}

Vertex-decomposability is interesting in this context since vertex-decomposable simplicial polytopes are known to satisfy the Hirsch bound~\cite[Corollary 2.11]{ProvanBillera}.
Observe that the paths stated in this theorem are within a factor of $1+\varepsilon$ from the upper bound of Theorem~\ref{thm:BarnetteLarman-intro}, with $\varepsilon$ going to zero as $N$ goes to infinity.

%%%%%%%%%%%%%%%%%%%%%%%%%%%%%%%%%%%%%%%%%%%%%%%%%%%%%%%%%%%%%%%%%%%%%%%%
%%%%%%% Combinatorial segments and monotone conservative paths %%%%%%%%%
%%%%%%%%%%%%%%%%%%%%%%%%%%%%%%%%%%%%%%%%%%%%%%%%%%%%%%%%%%%%%%%%%%%%%%%%

\section{Combinatorial segments and monotone conservative paths}
\label{sec:combinatorialsegments}

%%%%%%% Preliminaries on simplicial complexes %%%%%%%%%
%%%%%%%%%%%%%%%%%%%%%%%%%%%%%%%%%%%%%%%%%%%%%%%%%%%%%%%
\subsection{Preliminaries on simplicial complexes}

A simplicial complex, or simply a complex, on a \defn{vertex} set $V$ is a collection $C$ of subsets of $V$ such that $f\subset g\in C$ implies $f\in C$.
The elements of $C$ are called \defn{faces}. The~\defn{dimension} of a face~$f$ is~$|f|-1$.
If all its \defn{facets} (that is, inclusion maximal faces) have $d$ elements, $C$ is said to be \defn{pure} of dimension $d-1$, and sometimes called $(d-1)$-complex.
When describing a complex $C$ we usually only list its maximal faces.
For example, for $f\subset V$ and $v\in V$ we denote $\{f\}$ and $\{v\}$ the complexes $\{f': f'\subset f\}$ and $\{\{v\},\emptyset\}$, respectively.
Observe that from this perspective a \defn{pure simplicial $(d-1)$-complex} is nothing but a subset of $\binom{V}{d}$.
A \defn{minimal nonface} of a complex is a nonempty set of its vertices which is not a face but such that all its proper subset are faces.
A complex is called \defn{flag} if all its minimal nonfaces are 1-dimensional,  i.~e.,~edges.

The~\defn{star},~\defn{deletion} (sometimes called the \defn{antistar} or \defn{face-deletion}), and~\defn{link} of a face~$f$ in~$C$ are the following subcomplexes of~$C$, where $\sqcup$ denotes the disjoint union:
\begin{align*}
\st_C(f) & \eqdef \{f' \in C : f\cup f' \in C \}, \\
\del_C(f) &\eqdef \{f'\in C : f\not\subseteq f'\}, \\
\lk_C(f) & \eqdef \{f' \in C : f \sqcup f' \in C\}.
\end{align*}

The~\defn{join} of two simplicial complexes~$C_1$ and $C_2$ is the simplicial complex
\[
C_1\ast C_2\eqdef\{f_1\sqcup f_2 : f_1\in C_1,f_2\in C_2\}.
\]
\noindent Observe that $\st_C(f) = \lk_C(f) * \{f\}$.

The~\defn{suspension} of a simplicial complex~$C$ is the join of~$C$ with a simplicial complex consisting in two singletons.
Given a vertex~$v$ of a simplicial complex~$C$, the~\defn{one-point-suspension} of~$C$ with respect to~$v$ is the simplicial complex
\[
\ops_C(v)\eqdef (\lk_C(v) * \left\{\{v_1,v_2\}\right\})\ \cup\ (\del_C(v) * \left\{\{v_1\},\{v_2\}\right\}), 
\] 
where~$v_1$ and $v_2$ are new elements, called~\defn{suspension vertices}, that were not vertices of~$C$.
Equivalently, $\ops_C(v)$ is obtained contracting the edge $vv_2$ in the suspension $C * \{v_1,v_2\}$ of $C$.
Observe that 
\[
\lk_{\ops_{C}(v)}(v_1) \cong
\lk_{\ops_{C}(v)}(v_2) \cong C\qquad\text{ and }
\qquad
\lk_{\ops_{C}(v)}(v_1v_2) = \lk_C(v).
\]
One-point-suspensions are called \emph{wedges} in~\cite{ProvanBillera}.

Let $f$ be a face of $C$. The \defn{stellar subdivision} of $f$ in~$C$ is the simplicial complex denoted by~$\stsubd_C(f)$ and given by 
\[
	\stsubd_C(f)\eqdef\del_{C \cup ( \st_C(f)*\{a\})}(f)
=	\{f'\in C : f\not\subseteq f'\} \cup \left\{f'\cup\{a\} : f\not\subseteq f'\in \st_C(f) \right\},
\]
where~$a$ is a new vertex that is not a vertex of~$C$. Notice that in the particular case where~$f=\{v_1,\dots,v_d\}$ is a facet of~$C$, then~$\stsubd_C(f)$ consists of~$C$ in which the facet~$f$ has been replaced by the~$d$ facets obtained from~$f$ by substituting one of its vertices by~$a$.

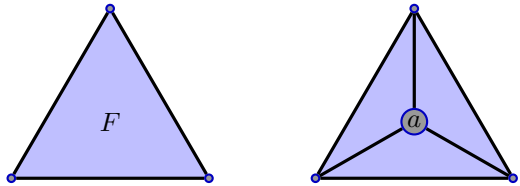
\begin{figure}[!htbp]
\centerline{\begin{tikzpicture}[vertex/.style={inner sep=1pt,circle,draw=blue!75!black,fill=black!40,thick}]

\def\rad{1.5}
\def\dist{4}

\coordinate (f1) at (-30:\rad);
\coordinate (f2) at (90:\rad);
\coordinate (f3) at (210:\rad);

\fill[blue!25!white] (f1) -- (f2) -- (f3) -- cycle;

\node at (0,0) {$F$};

\node[vertex] (nf1) at (f1) {};
\node[vertex] (nf2) at (f2) {} edge[very thick] (nf1);
\node[vertex] (nf3) at (f3) {} edge[very thick] (nf1) edge[very thick] (nf2);

\coordinate (g1) at ($(-30:\rad)+(\dist,0)$);
\coordinate (g2) at ($(90:\rad)+(\dist,0)$);
\coordinate (g3) at ($(210:\rad)+(\dist,0)$);
\coordinate (u) at (\dist,0);

\fill[blue!25!white] (g1) -- (g2) -- (g3) -- cycle;

\node[vertex] (nf1) at (g1) {};
\node[vertex] (nf2) at (g2) {} edge[very thick] (nf1);
\node[vertex] (nf3) at (g3) {} edge[very thick] (nf1) edge[very thick] (nf2);
\node[vertex] (nf4) at (u) {$a$} edge[very thick] (nf1) edge[very thick] (nf2) edge[very thick] (nf3);

\end{tikzpicture}}
\caption{The stellar subdivision of a simplex $F$.}
\label{fig:stellar}
\end{figure}

A pure $(d-1)$-complex $C$ is \defn{vertex-decomposable}~\cite{DeloeraKlee,ProvanBillera} if it is either a simplex or if there exists a vertex $x\in C$ such that $\lk_{C}(x)$ and $\del_C(x)$ are both vertex-decomposable and $\del_C(x)$ is still pure of dimension $d-1$.

\begin{lemma}[{\cite[Proposition~2.5 and Theorem~2.7]{ProvanBillera}}]
\label{lem:ops-vertex-decomposable}
Let~$C$ be a~$d$-complex,~$x$ a vertex of~$C$ and~$f$ a face of~$C$. 
\begin{itemize}
 \item The complex $C$ is vertex-decomposable if and only if $\ops_C(x)$ is vertex-decomposable.
 \item If $C$ is vertex-decomposable, then $\stsubd_C(f)$ is a vertex-decomposable~$d$-complex.
\end{itemize}
\end{lemma}

The \defn{facet-adjacency graph} or \defn{dual graph} of $C$ has the facets of $C$ as nodes with two facets $F$ and $F'$ being adjacent in the graph if they differ in a single vertex.
A pure simplicial complex~$C$ is called \defn{normal} if between every two facets $F$, $F'$ there is a~\defn{dual path} (a path in the dual graph) consisting only of facets that contain $F\cap F'$.
A pure simplicial complex $C$ is called a \defn{pseudomanifold} if every codimension-one simplex lies in either 1 or 2 facets.
The codimension-one simplices lying in only one facet form the \defn{boundary} of $C$.

Although our main question of interest is the diameter of the dual graph, we use also the distance in the ``usual graph'' (or $1$-skeleton) of $C$, whose vertices and edges are those of~$C$ itself. We refer to it as the \defn{vertex-distance} between two vertices $u$ and $v$ of $C$ and denote it $\vdist_C(u,v)$. For two sets of vertices $S$ and~$S'$ we denote
\[
\vdist_C(S,S'):= \min_{u\in S, v\in S'} \vdist_C(u,v),
\]
with the standard convention that if one or both of $S$ and $S'$ are empty the minimum is $\infty$. If~$\dim(C)=0$, we also take the convention that for nonempty sets~$S$ and~$S',\vdist_C(S,S')=0$ if~$S\cap S'\neq\emptyset$, and~$\vdist_C(S,S')=1$ otherwise.

%%%%%%% Monotone conservative paths %%%%%%%%%
%%%%%%%%%%%%%%%%%%%%%%%%%%%%%%%%%%%%%%%%%%%%%
\subsection{Monotone conservative paths}

The following definitions are crucial in order to define monotonicity and conservativeness of paths.

\begin{definition}
\label{def:ordered-facet}
\label{defi:vector-of-distances}
Let $F=(v_1,\dots,v_d)$ be a facet of a pure $(d-1)$-dimensional simplicial complex~$C$ with its vertices given in a specific order.
We call this an \defn{ordered facet}.
Let $S$ be a subset of vertices of $C$, refered to as the~\defn{target set}.
The~\defn{vector of distances} $\Lambda=(\lambda_1,\dots,\lambda_d)$ from $F$ to $S$ is defined as follows:
\begin{itemize}
\item $\lambda_1:= \vdist(v_1, S)$.
\item $(\lambda_2,\dots,\lambda_d)$ is the vector of distances from $F\ssm v_1$ to $S'$ in $\lk_C(v_1)$, where $S'$ is the following subset of $V_1:=\operatorname{vertices}(\lk_C(v_1))$:
\[
S':=
\begin{cases}
\{v \in V_1 : \vdist_C(v,S) = \vdist_C(v_1, S)-1\}  & \text{if } v_1\not\in S,\\
S\cap \lk_C(v_1) & \text{if } v_1\in S.
\end{cases}
\]
\end{itemize}
We say the ordering is \defn{admissible} with respect to $S$ if $\lambda_1=\vdist(v_1,S)=\vdist(F,S)$ and $F\ssm v_1$ (with the induced ordering) is admissible with respect to  $S'$ in $\lk_C(v_1)$.
\end{definition}

\begin{example}
Consider the pure flag 2-dimensional complex given in Figure~\ref{fig:ex_admissible}, with target set~$S=\{s_1,s_2\}$. The ordered facet $F_1=(a_1,a_2,a_3)$ is not admissible, whereas the ordered facet $F_2=(b_1,b_2,b_3)$ and $F_3=(c_1,c_2,c_3)$ are admissible. The facet $F_1$ would be admissible with respect to other orderings, e.~g.~$(a_2,a_3,a_1)$. The vectors of distances for the facets $F_1,F_2$ and $F_3$ are $(5,0,0)$, $(2,1,1)$, and $(0,0,\infty)$ respectively. In particular, it is possible to have infinite values in a vector of distances. 
\end{example}

\begin{figure}[!htbp]
 \centerline{\begin{tikzpicture}[vertex/.style={inner sep=1pt,circle,draw=blue!75!black,fill=black!40,thick}]

\coordinate (rowdown) at (0,0);
\coordinate (rowup) at (60:1);
 
\node[vertex,label=below:{$a_1$}] (d1) at (rowdown) {};
\node[vertex,label=below:{$a_2$}] (d2) at ($(rowdown)+(1,0)$) {} edge[very thick] (d1);
\node[vertex,label=below:{$b_3$}] (d3) at ($(rowdown)+(2,0)$) {} edge[very thick] (d2);
\node[vertex,label=below:{$b_1$}] (d4) at ($(rowdown)+(3,0)$) {} edge[very thick] (d3);
\node[vertex] (d5) at ($(rowdown)+(4,0)$) {} edge[very thick] (d4);
\node[vertex,label=below:{$s_1=c_1$}] (d6) at ($(rowdown)+(5,0)$) {} edge[very thick] (d5);

\node[vertex,label=above:{$a_3$}] (u1) at (rowup) {} edge[very thick] (d1) edge[very thick] (d2);
\node[vertex] (u2) at ($(rowup)+(1,0)$) {} edge[very thick] (d2) edge[very thick] (d3) edge[very thick] (u1);
\node[vertex,label=above:{$b_2$}] (u3) at ($(rowup)+(2,0)$) {} edge[very thick] (d3) edge[very thick] (d4) edge[very thick] (u2);
\node[vertex] (u4) at ($(rowup)+(3,0)$) {} edge[very thick] (d4) edge[very thick] (d5) edge[very thick] (u3);
\node[vertex,label=above:{$s_2=c_2$}] (u5) at ($(rowup)+(4,0)$) {} edge[very thick] (d5) edge[very thick] (d6) edge[very thick] (u4);
\node[vertex,label=above:{$c_3$}] (u6) at ($(rowup)+(5,0)$) {} edge[very thick] (d6) edge[very thick] (u5);

\node at ($(d1)!0.5!(d2)!0.33!(u1)$) {$F_1$};
\node at ($(d3)!0.5!(d4)!0.33!(u3)$) {$F_2$};
\node at ($(u5)!0.5!(u6)!0.33!(d6)$) {$F_3$};

\end{tikzpicture}}
 \caption{Examples of admissible and nonadmissible ordered facets in a~$2$-dimensional complex.}
 \label{fig:ex_admissible}
\end{figure}
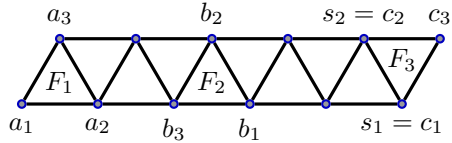

We are primarily interested in the case where $S$ is a facet, but the recursive definition forces us to consider more general target sets. Observe that the definition may even make the new target set $S'$ obtained from $S$ be empty, in which case all entries in the vector of distances starting at that point will be infinite. The following statement shows that this  never happens in a normal complex if the initial target set is a facet.

\begin{lemma}
\label{lem:finiteEntries}
Let~$C$ be a normal~$(d-1)$-complex, $S$ be a nonempty set of vertices of~$C$, and~$F=(v_1, \dots, v_d)$ be an ordered facet of $C$.
If $F$ is admissible with respect to~$S$ and that one of the following condition is true
 \begin{itemize}
  \item $F\cap S=\emptyset$,
  \item $S$ is a face of~$C$ not strictly contained in $F$,
 \end{itemize} 
 then the vector of distances~$\Lambda=(\lambda_1,\dots,\lambda_d)$ from~$F$ to~$S$ has no infinite entry.
\end{lemma}

\begin{proof}
We use induction on the dimension~$d-1$ of~$C$. For~$d=1$, in both cases the vector of distances has a single entry which is either zero or one. Indeed the set~$F$ is a singleton, and~$S$ is a nonempty set.
For~$d>1$, the entry~$\lambda_1$ of the vector of distances is finite since~$S$ is nonempty and~$C$ is normal thus has connected~$1$-skeleton.
Notice that the simplicial complex~$C$ being pure and normal ensures that~$\lk_C(v_1)$ is also pure, normal and of dimension one less.
Further, as the ordering on~$F$ is admissible with respect to~$S$, so is the ordering on~$F\ssm v_1$ with respect to~$S'$.
Suppose first that~$F\cap S=\emptyset$. Because the ordering on~$F$ is admissible with respect to~$S$, the vertex~$v_1$ is at smallest (and nonzero) vertex-distance of~$S$ in the facet~$F$.
Therefore the facet~$F\ssm v_1$ and the set~$S'$ in Definition~\ref{def:ordered-facet} are disjoint. Suppose now that~$S$ is a face of~$C$ such that~$v_1\in S$.
In this case~$S'=S\cap\lk_C(v_1)$ is a face of~$\lk_C(v_1)$ 
which is not strictly contained in~$F\ssm v_1$.
In both cases, the induction hypothesis ensures that the vector~$(\lambda_2,\dots,\lambda_d)$ has no infinite entry.
\end{proof}

The previous proof also shows that if~$F$ is admissible with respect to~$S$ and the vector of distances contains some zero, then all zeroes are at the beginning of the vector. That is to say:
no entry of its vector of distances may be zero if a previous entry is nonzero.

We now introduce the main definitions used in this article.

\begin{definition}
\label{def:monotone-path}
Let~$C$ be a pure and normal complex, $S$ be a nonempty target set of vertices of~$C$, and $\Gamma=[F_0,\dots,F_N]$ be a path of ordered facets in the dual graph of $C$.
\begin{enumerate}
	\item We say that $\Gamma$ is \defn{(combinatorially) monotone towards $S$} if the sequence $(\Lambda_0,\dots,\Lambda_N)$ of vectors of distances from its facets to~$S$ is lexicographically decreasing.
That is, $\Lambda_{i+1}$ is lexicographically smaller than $\Lambda_i$ for every $i$.
	\item For each $i=1,\dots,N$ we call \defn{index} of the step from $F_{i-1}$ to $F_i$ in $\Gamma$ the minimum $k$ for which $v_k\ne v'_k$, where $F_{i-1}=(v_1,\dots,v_d)$ and $F_{i}=(v'_1,\dots,v'_d)$. Alternatively we call it sometimes the index of~$F_i$ when there is no ambiguity on the path~$\Gamma$ that we are considering. We say that the step from $F_{i-1}$ to $F_i$, is \defn{conservative towards $S$} if the following two things happen:
	\begin{itemize}
		\item $F_{i-1}\ssm F_i$ is the last vertex in the ordering of $F_{i-1}$. That is, $F_i$ keeps the initial vertices of $F_{i-1}$ as much as possible.
		\item $F_i$  is admissible with respect to $S$ and it has the maximum index among all possible choices of admissible reorderings of $F_i$.  
		That is, $F_i$ tries to keep the order established in $F_{i-1}$ as much as admissible.
	\end{itemize}
	We say that $\Gamma$ itself is \defn{conservative towards $S$} if $F_0$ is admissible and every step is conservative.
\end{enumerate}
\end{definition}

\begin{example}\label{ex:non_ex}
Consider the pure flag 2-dimensional complex given in Figure~\ref{fig:ex_monotone_conserv}, with target set $S=\{a_8\}$.
The ordered facet $F=(a_1,a_2,a_3)$ has distance vector $(2,2,1)$ and is admissible.
The ordered facet $G_1=(a_5,a_1,a_3)$ and $G_2=(a_1,a_5,a_3)$ both have distance vector $(2,1,1)$ and are admissible.
The ordered facet $H_1=(a_4,a_1,a_2)$ and $H_2=(a_1,a_4,a_2)$ both have distance vector $(2,1,1)$ and are admissible.
The step $[F,G_1]$ is monotone, but does not satisfy either conditions of conservativeness: $a_2$ is not the last vertex of~$F$ and $G_1$ does not have maximal index.
The step $[F,G_2]$ is monotone but not conservative: $G_2$ has maximal index, but $a_2$ is not the last vertex of~$F$.
The step $[F,H_1]$ is monotone but not conservative: $a_3$ is the last vertex of $F$, but $H_1$ is not of maximal index.
The step $[F,H_2]$ is monotone and conservative: $a_3$ is the last vertex of $F$ and $H_2$ has maximal index.
Note that these four steps form initial segments of  shortest paths in the dual graph from $F$ to $S$, but only one of them is monotone and conservative.

To obtain a nonmonotone but conservative step, consider the complex $\left\{\{1,2\},\{1,3\},\{1,4\}\right\}$ with $F=(1,2)$, $G=(1,3)$, and $S=\{4\}$.
The path $[F,G]$ is not monotone but is conservative.
\end{example}

\begin{figure}[!hbtp]
 \centerline{\begin{tikzpicture}[vertex/.style={inner sep=1pt,circle,draw=blue!75!black,fill=black!40,thick}]

\def\rad{1.5}

\node[vertex,label=below:{$a_1$}] (a1) at (0,0) {};
\node[vertex,label=below:{$a_2$}] (a2) at (240:\rad) {} edge[very thick] (a1);
\node[vertex,label=below:{$a_3$}] (a3) at (300:\rad) {} edge[very thick] (a1) edge[very thick] (a2);
\node[vertex,label=left:{$a_4$}] (a4) at (180:\rad) {} edge[very thick] (a1) edge[very thick] (a2);
\node[vertex,label=right:{$a_5$}] (a5) at (0:\rad) {} edge[very thick] (a1) edge[very thick] (a3);
\node[vertex,label=left:{$a_6$}] (a6) at (120:\rad) {} edge[very thick] (a1) edge[very thick] (a4);
\node[vertex,label=right:{$a_7$}] (a7) at (60:\rad) {} edge[very thick] (a1) edge[very thick] (a5) edge[very thick] (a6);
\node[vertex,label=above:{$a_8$}] (a8) at ($(60:\rad)+(120:\rad)$) {} edge[very thick] (a6) edge[very thick] (a7);

\node at ($(a2)!0.5!(a3)!0.33!(a1)$) {$F$};
\node at ($(a1)!0.5!(a5)!0.33!(a3)$) {$G$};
\node at ($(a1)!0.5!(a4)!0.33!(a2)$) {$H$};

\end{tikzpicture}}
 \caption{A $2$-dimensional flag simplicial complex having monotone and nonconservative paths.}
 \label{fig:ex_monotone_conserv}
\end{figure}
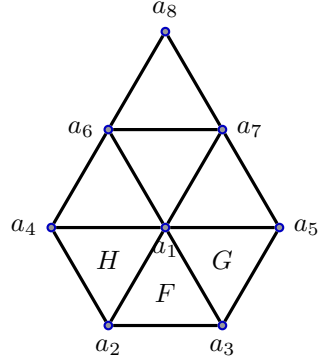

\begin{remark}
In a conservative step~$[F_{i-1},F_i]$ towards a target set~$S$, the index somehow denotes the ``depth'' at which the step occurs in the recursive definition.
\end{remark}

It is clear from this definition that for a path~$\Gamma$ of ordered facets in the dual graph and a target set~$S$,
\begin{itemize}
 \item if~$\Gamma$ is monotone (resp.~conservative) towards~$S$, then any subpath of~$\Gamma$ is also monotone (resp.~conservative) towards~$S$,
 \item if~$\Gamma_1$ and $\Gamma_2$ are monotone (resp.~conservative) towards~$S$ and the last ordered facet in $\Gamma_1$ equals the first one in $\Gamma_2$, then the concatenation of $\Gamma_1$ and $\Gamma_2$ is monotone (resp.~conservative) towards~$S$. 
\end{itemize}

Define the~\defn{anchor} of an ordered facet to be its first vertex. The anchors along a monotone path towards~$S$ form a vertex-path in the~$1$-skeleton of~$C$ along which the distance to~$S$ decreases.
In particular, they form a shortest path between the first and last anchor. Notice that in a monotone and conservative path~$\Gamma$ from~$F$ towards~$S$, the simplices with the same anchor~$x_0$ form a subpath~$\Gamma_0$ of~$\Gamma$. Moreover, 
the deletion of~$x_0$ in the path~$\Gamma_0$ is a monotone conservative path~$\Gamma'_0$ in~$\lk_C(x_0)$ towards the set~$\{x\in V : \vdist_C(x,x_0)=1,\vdist_C(x,S) = \vdist_C(x_0,S)-1\}$ which is the set~$S'$ of Definition~\ref{def:ordered-facet}.

We now show the existence of monotone conservative paths in any pure normal complex, and bound their lengths. The following theorem deals with the existence part, up to the (immediate) existence of admissible orderings.

\begin{theorem}
\label{thm:step-existence}
Let $C$ be a pure and normal complex, $S$ be a nonempty target set, and $F_0$ be an ordered facet of~$C$.
If $F_0\cap S=\emptyset$ and $F_0$ is admissible with respect to~$S$, then there exists an ordered facet $F_1$ adjacent to~$F_0$ such that the step $[F_0,F_1]$ is monotone and conservative.
\end{theorem}

\begin{proof}
Lemma~\ref{lem:finiteEntries} ensures that all entries in the vector of distances of~$F_0$ are finite, so that we do not need to define any comparison convention on potential infinite values.

The proof works by induction on the dimension~$d-1$ of~$C$, the case~$d=1$ being clear.

Let us also deal with the case $d=2$ separately.
In this case, $C$ is a graph and $F_0$ is an ordered edge $(u,v)$. The vector of distances is $(a,1)$, where $a\ge 1$ is the distance from~$u$ to~$S$ in the graph. Simply take $F_1=(w,u)$, where $w$ is any neighbor of $u$ at distance $a-1$ from $S$.

For arbitrary $d$, let $F_0=\{v_1,\dots, v_d\}$ (in this order) and let $F'_0=\{v_2,\dots, v_{d}\}=F_0\ssm v_1$.
Observe that (as in the proof of Lemma~\ref{lem:finiteEntries}) the condition $F_0\cap S=\emptyset$ is inherited into $F'_0 \cap S' = \emptyset$, where 
\[
S'=\{v \in \operatorname{vertices}(\lk_C(v_1)) : \vdist_C(v,S) = \vdist_C(v_1, S)-1\}
\]
is as in the definition of admissibility. In particular, by inductive hypothesis, there is an $F'_1$ containing $F'_0\ssm\{v_d\}$
such that (with a suitable ordering for $F'_1$) $[F'_0,F'_1]$ is a monotone conservative path from $F_0$ to $S'$ in $\lk_{v_1}(C)$. Let $F_1=\{v_1\}\cup F'_1$ considered, for now, with $v_1$ as the first vertex, followed by the rest in the order of $F'_1$.
Then:
\begin{itemize}
\item The vector of distances of $F_1$ with respect to $S$ has the same first entry as that of $F_0$, and the rest of entries are lexicographically smaller than those of $F_0$ by inductive hypothesis. Thus, $[F_0,F_1]$ is monotone.
\item Again by inductive hypothesis, $F_{0}\ssm F_1$ is the last vertex in the ordering of $F_0$.
\end{itemize}
If $\vdist(F_1,S) = \vdist(v_1,S)$, the ordering of $F_1$ is admissible and, by inductive hypothesis, it has the maximum index among admissible orderings; hence $[F_0,F_1]$ is admissible and conservative. If $\vdist(F_1,S) < \vdist(v_1,S)$, the first vertex of any admissible ordering of~$F_1$ cannot be~$v_1$, therefore the index of this step has to be~$0$. Then choosing any reordering for $F_1$ that is admissible makes the step~$[F_0,F_1]$ conservative. It is also still monotone since the first entry of the vector of distances of~$F_1$ is then smaller than the first entry of the vector of distances of~$F_0$.
\end{proof}

Observe that the reordering of $F_1$ in the last part of the previous proof produces a change of anchor, from $v_1$ to $v'_1$. In fact, every change of anchor is produced by such a reordering. When an anchor is changed in the step from $F_0$ to $F_1$ (say from an anchor $x$ to an anchor $y$), vertex $x$ is still in the new facet $F_1$ because the vertex in $F_0\ssm F_1$ must be the last vertex in the order of $F_0$. So, the change of anchor is not due to the disappearance of $x$ from the facet, but rather to the fact that $x$ is no longer a closest vertex to $S$ in $F_1$ and a reordering is needed.

\begin{corollary}
\label{coro:path-existence}
Let $C$, $S$ and $F_0$ be as in Theorem~\ref{thm:step-existence}.
There exists a monotone conservative path towards $S$ starting at $F_0$ and ending in a facet that meets $S$.
\end{corollary}

\begin{proof}
If $F_0$ meets $S$ then $F_0$ alone is the desired path. If $F_0\cap S=\emptyset$ then we apply  Theorem~\ref{thm:step-existence}, which gives us an $F_1$ to start the path. 
\end{proof}

\begin{corollary}
\label{coro:path-existence-2}
Let $C$, $S$ and $F_0$ be as in Theorem~\ref{thm:step-existence}.
If $S$ is the vertex set of a facet, then there is a monotone conservative path towards $S$ starting at~$F_0$ and with $S$ as its last facet.
\end{corollary}

\begin{proof}
By the previous corollary, there is no loss of generality in assuming $F_0\cap S\ne \emptyset$.
That is, $v_1\in S$. Now, $S':=S\ssm \{v_1\}$ is a facet in $\lk_{C}(v_1)$. 
By induction on the dimension, there is a monotone path in $\lk_{C}(v_1)$ from $F'_0:=F_0\ssm \{v_1\}$ which is conservative towards $S':=S\ssm \{v_1\}$ and ending in~$S'$.
Adding $v_1$ as the first vertex in all  facets of that path gives the desired path from $F_0$ to~$S$.
\end{proof}

Theorem~\ref{thm:step-existence} also has the following consequence:

\begin{proposition}
\label{prop:pseudomanifold}
In a pseudomanifold (with or without boundary), a conservative path is automatically monotone.
\end{proposition}

\begin{proof}
Let $[F_0,F_1]$ be a conservative step in a pseudomanifold.
If $F_0\cap S\ne \emptyset$, then we induct on the dimension, by taking the link of the first vertex $v_1$ of $F_0$, which is still a pseudomanifold.
So we can assume that $F_0\cap S = \emptyset$.

By Corollary~\ref{coro:path-existence-2}, from $F_0$ there is a monotone and conservative step to a certain facet $F'_1$.
Now, since the definition of conservativeness determines which vertex of $F_0$ is to be removed, and since we are in a pseudomanifold, we have $F_1=F'_1$ up to reordering.
Being conservative implies that the first vertex in which the orderings of $F_1$ and $F'_1$ differ is also the first vertex in which they differ from $F_0$, and that the corresponding entry in the vector of distances of $F_1$ and $F'_1$ is smaller than the same entry in the vector of distances of $F_0$.
Hence, the step $[F_0,F_1]$ is monotone.
\end{proof}

%%%%%%% Combinatorial segments as monotone conservative paths %%%%%%%%%
%%%%%%%%%%%%%%%%%%%%%%%%%%%%%%%%%%%%%%%%%%%%%%%%%%%%%%%%%%%%%%%%%%%%%%%
\subsection{Combinatorial segments as monotone conservative paths}

We now relate these concepts with the notion of combinatorial segment by Adiprasito and Benedetti~\cite{AdiprasitoBenedetti}.
The following definition is taken from~\cite{Santos-progress}.

\begin{definition} 
\label{definition:combinatorial-segment}
A dual path~$\Gamma=[F_0,\dots,F_N]$ with~$x_0\in F_0$ in a simplicial complex $C$
is a~\defn{combinatorial segment from a facet~$F$ to a set~$S$ anchored at the vertex~$x_0$} if it satisfies the following recursive definition:

\begin{enumerate}[1)]
 \item if~$F_0\cap S\neq\emptyset$, then~$N=0$,
 \item if~$d=1$ and~$F_0\cap S=\emptyset$, then~$N=1$ and~$\Gamma=(\{x_0\},\{v\})$ with~$v\in S$,
 \item if~$d>1$ and~$F_0\cap S=\emptyset$, then:
	 \begin{enumerate}[a)]
          \item the facet~$F_N$ is the unique facet of~$\Gamma$ intersecting~$S$,

          \item let~$\ell=\vdist_C(F_0,S)$, let~$k=\min\{i\in[N]|\vdist_C(F_i,S)<\ell\}$ and let~$y$ be the unique vertex in~$F_k$ such that~$\vdist_C(\{y\},S)=\ell-1$. Then~$x_0\in F_0\cap\dots\cap F_k$ and the link~$\Gamma'_1$ of~$x_0$ in the path~$\Gamma_1=(F_0,F_1,\dots,F_k)$ is a combinatorial segment in~$\lk_C(x_0)$ from the facet~$F_0\ssm x_0$ to the set of neighbors of~$x_0$ being at vertex-distance~$\ell-1$ from~$S$ in~$C$ (that is the set~$S'$ of Definition~\ref{def:ordered-facet} of admissibility),

          \item the path~$\Gamma_2=[F_k,\dots,F_N]$ is a combinatorial segment from~$F_k$ to~$S$ anchored at~$y$.
         \end{enumerate}
\end{enumerate}
\end{definition}

In~\cite{Santos-progress}, the anchor of the path~$\Gamma$ is the vertex~$x_0$.
In our definition $x_0$ is only the ``first anchor'' of the path, after which come all the anchors in the path $\Gamma_2$.
Anchors, in the sense of~\cite{Santos-progress}, are called \emph{pearls} in~\cite{AdiprasitoBenedetti}.

Observe that the facets of a combinatorial segment come with an implicit ordering of their vertices. For the facets in $\Gamma_1$, except the last one $F_k$, $x_0$ is the first vertex in the ordering and the rest of the ordering is obtained by induction on the dimension. For facets in $\Gamma_2$, the ordering is obtained by induction on $N$ (or, alternatively, on the vertex-distance from $F_0$ to $S$).
These orderings make combinatorial segments and monotone conservative paths essentially equivalent:
 
\begin{theorem}
\label{thm:equivalenceDefinitions}
Let~$C$ be a simplicial complex, $S$ be a nonempty set of vertices of~$C$, $\Gamma=[F_0,\dots,F_N]$ be a dual path in~$C$ in which $F_N$ is the only facet intersecting $S$.
The path~$\Gamma$ is a combinatorial segment from $F_0$ to $S$ if and only if it is monotone and conservative towards $S$.
\end{theorem}

\begin{proof}
 We make the proof follow the inductive definition of combinatorial segment. Observe that it has a double induction, first on the dimension and then on $N$ (or, alternatively, on $\vdist(F_0,S)$).
 
 In case (1) of the definition, the result trivially holds. In case (2) observe that for $d=1$ the ``vector of distances'' between a vertex $F$ and a set $S$ is just a number, zero if $F\in S$ and one if not. Since monotonicity implies these numbers to be decreasing, a monotone path must have $N=1$ and be of the form $[\{x_0\},\{v\}]$ with~$v\in S$.
  
 So suppose that~$F_0\cap S=\emptyset$ and that~$d>1$. Assume for now that~$\Gamma$ is a combinatorial segment from~$F_0$ to~$S$. Then we already know that the path~$\Gamma_2$ of the definition is also a combinatorial segment from~$F_k$ to~$S$, implying by induction that any step of this path is both monotone and conservative towards~$S$.
 By the induction hypothesis, since the path~$\Gamma'_1$ of the definition is a combinatorial segment, it is monotone and conservative towards the set~$S'$ of Definition~\ref{defi:vector-of-distances}. The vectors of distances in the path~$\Gamma_1$ are all, except for the one in the last facet $F_k$, obtained from those in the path~$\Gamma'_1$ by adding the distance of the anchor~$x_0$ as first coordinate, so that~$\Gamma_1$ is still monotone. 
 Moreover, the dual graph of~$\lk_C(x_0)$ is the same as that of~$\st_C(x_0)$, so that conservativeness is also preserved from~$\Gamma_1$ to~$\Gamma'_1$ and conversely, except perhaps for the last step. Monotonicity in the step $[F_{k-1},F_k]$ is clear, since $\vdist(y,S) < \vdist(x_0,S)$. Let us check the conditions for conservativeness:
 \begin{itemize}
 \item The vertex in $F_{k-1}\ssm F_k$ is the last vertex of $F_{k-1}$ because it is the last vertex in $F'_{k-1}:=F_{k-1}\ssm \{x_0\}$ (induction on the dimension).
 \item The ordering in $F_k$ is admissible and has maximal index among admissible ones because $\Gamma_2$ is monotone and conservative, by induction on $N$.
 \end{itemize}
 
 Since $\Gamma_2$ is also monotone and conservative, the path $\Gamma$ is monotone and conservative.
 
 Conversely, assume now that~$\Gamma$ is monotone and conservative towards~$S$. Let~$x_0$ be the first anchor of~$\Gamma$ and~$[F_0,\dots,F_{k-1}]$ be the part of $\Gamma$ anchored at $x_0$ and let
 $\Gamma_1=[F_0,\dots,F_{k}]$. Observe that, by conservativeness, $F_k$ still contains $x_0$. Moreover, 
 the link $\Gamma'_1$ of~$x_0$ in~$\Gamma_1$ is monotone and conservative towards~$S'$ except for the last step. In this last step, the order on~$F_k$ in~$\Gamma$ cannot be restricted to an ordering of $F'_k:=F_k\ssm \{x_0\}$ in $\lk_{C}(x_0)$, since $x_0$ is not the first vertex in $F_k$. 
Nevertheless the same arguments as above still apply and any admissible ordering on~$F_k$ with respect to~$S'$ with highest index guarantees  monotonicity and conservativeness for the step~$[F_{k-1},F_k]$. 
So~$\Gamma'_1$ is a combinatorial segment to~$S'$ by the induction hypothesis.
Moreover, the path~$\Gamma_2=[F_k,\dots,F_N]$ is a subpath of~$\Gamma$ and so is monotone and conservative towards~$S$, and so is a combinatorial segment from~$F_k$ to~$S$ by induction on $N$. Thus~$\Gamma$ is a combinatorial segment from~$F_0$ to~$S$ in~$C$.
\end{proof}

Adiprasito and Benedetti also define a combinatorial segment ``between two facets'' as follows:
Let $F,F'$ be two facets of $C$ and let $\Gamma$ be a combinatorial segment from the facet $F$ to the \emph{set of vertices} $F'$. Then, only the last facet $F''$ in $\Gamma$ intersects $F'$, and $F''$ and $F'$ intersect in a single vertex~$x$. Since $F''\ssm x$ and $F'\ssm x$ are both facets in $\lk_C(x)$ we can recursively define a combinatorial segment from the facet $F$ to the facet $F'$ to be $\Gamma$ followed by $x*\Gamma'$, where $\Gamma'$ is a combinatorial segment in $\lk_C(x)$ from the facet $F''\ssm x$ to the facet $F'\ssm x$.

The following result follows easily:

\begin{corollary}
Let $F,F'$ be facets in a normal complex $C$.
A combinatorial segment from~$F$ to~$F'$ is a monotone and conservative path that starts in~$F$ towards the target set~$F'$ and finishes in~$F'$.
\end{corollary}

\begin{proof}
	By Theorem~\ref{thm:equivalenceDefinitions}, the combinatorial segment starting at $F$ and ending at the first facet intersecting $F'$ at the vertex $x$ is monotone and conservative towards the vertex set of $F'$. By induction on the dimension, we join $x$ to the monotone and conservative path in the link of $x$ and concatenate it to the previous path. The result is clearly monotone and conservative.
\end{proof}

%%%%%%% Joins and one-point-suspensions %%%%%%%%%
%%%%%%%%%%%%%%%%%%%%%%%%%%%%%%%%%%%%%%%%%%%%%%%%%
\subsection{Joins and one-point-suspensions}

Let $\maxlength(C)$ denote the maximum length of monotone conservative paths in a pure normal complex $C$.

\begin{proposition}
 The function $\maxlength$ has the following properties.
 \begin{itemize}
  \item[(i)] $\maxlength(C_1\join C_2)=\maxlength(C_1)+\maxlength(C_2)$
  \item[(ii)] $\maxlength(\ops_C(v))\in\{\maxlength(C),\maxlength(C)+1\}$
 \end{itemize}
\end{proposition}

\begin{proof}
(i) It is clear that every dual path $\Gamma$ in $C_1\join C_2$ restricts as a path $\Gamma_1$ in $C_1$ and a path $\Gamma_2$ in $C_2$ so that $\length(\Gamma_1)+\length(\Gamma_2)=\length(\Gamma)$.
To prove  $\maxlength(C_1\join C_2)\leq\maxlength(C_1)+\maxlength(C_2)$ we only need to check that if $\Gamma$ is monotone and conservative, then $\Gamma_1$ and $\Gamma_2$ are also monotone and conservative.

Let $\Gamma$ be a monotone and conservative path from the facet  $F=F_1 \join F_2$ to the facet $G=G_1 \join G_2$, where $F_1,G_1$ (resp. $F_2,G_2$) are facets of $C_1$ (resp. $C_2$). 
Observe that flipping a vertex in $C_1$ does not influence the vertex-distances from vertices of a facet in $C_1$ to the facet $G_2$.
Therefore the restriction of $\Gamma$ to $C_1$ always flips according to vertex-distances within $C_1$.
Further, the flip does not depend on the position of the anchor, i.e., it occurs either in $C_1$ or in $\lk_{C_1}(v)$, which respects the definition of vector of distances in $C_1$. Therefore, the restriction of $\Gamma$ to $C_1$ (and symmetrically to $C_2$) forms a monotone and conservative path. 
 
Conversely, taking monotone and conservative paths in $C_1$ and $C_2$ determines a monotone and conservative path in $C_1\join C_2$ by shuffling the consecutive subsegments depending on the minimal vertex-distances of the anchors. Therefore $\maxlength(C_1)+\maxlength(C_2)\leq\maxlength(C_1\join C_2)$, which concludes.
 
(ii) Since $C$ is isomorphic to the link of $v_1$ in $\ops_C(v)$, we have $\maxlength(\ops_C(v)) \ge \maxlength(C)$. So, we only need to check that $\maxlength(\ops_C(v)) \le \maxlength(C) +1$.
Let $\Gamma$ be a longest monotone conservative path in the one-point-suspension $\ops_C(v)$ and let $\Gamma'$ be the path restricted to $C$ in the following sense: facets of $\ops_C(v)$ of the type $F*v_1$ or $F*v_2$ are sent to $F$, facets $F*\{v_1,v_2\}$ are sent to $F*v$.
If one of $v_1$ or $v_2$ (say $v_1$) belongs to all the facets in $\Gamma$, then $\Gamma'$ is a monotone and conservative path in $\lk_{\ops_{C}(v)}(v_1)$, which is isomorphic to $C$. 
If none of $v_1$ or $v_2$ belongs to all the facets in $\Gamma$, then assume that the starting facet contains $v_1$ but not $v_2$ and the final facet contains $v_2$ but not~$v_1$.
The path $\Gamma'$ does the same flips as $\Gamma$ except for the flips that change $v_1$ to $v_2$, or vice-versa.
We need to show that there is only one of those.
Since the stars of $v_1$ and $v_2$ contain all vertices, a step $[F,G]$ flipping $v_1$ to $v_2$ (or vice-versa) occurs with a vector of distances $(1,\dots,1)$ for $F$.
Therefore, $G$ has to contain a vertex in the original target set $S$.
Since $v_2$ is the only vertex not in~$F$, we have $v_2\in S$.
Then, by the recursive definition of monotone conservative paths, the remaining part of the monotone conservative path in $\ops_C(v)$ is formed in $\lk_{\ops_C(v)}(v_2)$ which is isomorphic to $C$ and $v_2$ is contained in the rest of the monotone conservative path.
It remains to show that $\Gamma'$ is monotone and conservative.
The path $\Gamma'$ consists of two parts, the one containing~$v_1$ and the one containing $v_2$. By the argument above they are both monotone and conservative with respect to a same target set so concatenating them gives a monotone and conservative path up to the repetition of a unique facet of $C$.
\end{proof}

%%%%%%%%%%%%%%%%%%%%%%%%%%%%%%%%%%%%%%%%%%%%%%%%%%%%%%%%%%%%%%%%%%%%%%%%%%%%
%%%%%%% Upper bounds for the length of monotone conservative paths %%%%%%%%%
%%%%%%%%%%%%%%%%%%%%%%%%%%%%%%%%%%%%%%%%%%%%%%%%%%%%%%%%%%%%%%%%%%%%%%%%%%%%
\section{Upper bounds for the length of monotone conservative paths}
\label{sec:upperbounds}

The motivation for the definition of combinatorial segment was the proof of the Hirsch upper bound for flag normal complexes, and it gave as a byproduct the linear bound in fixed dimension. In this section we rework these two bounds in the language of monotone conservative paths.
We also include a bound for banner complexes that interpolates between the two.

%%%%%%% Hirsch bound in flag normal complexes %%%%%%%%%
%%%%%%%%%%%%%%%%%%%%%%%%%%%%%%%%%%%%%%%%%%%%%%%%%%%%%%%
\subsection{Hirsch bound for flag normal complexes}
\label{subsection:nonrevisiting}

A path $\Gamma = [F_0,\dots, F_N]$ in a complex $C$ is said to be \defn{nonrevisiting} if 
for every vertex $v$, the facets of $\Gamma$ containing $v$ form a subpath of $\Gamma$. That is to say, if $v\in F_i\cap F_j$ then $v\in F_k$ for every $k\in[i,j]$. It is easy to show (see below) that the length of nonrevisiting paths in a pure complex cannot exceed the  Hirsch bound $n-d$. 

\begin{lemma}[{\cite[Section~3]{AdiprasitoBenedetti}, see also~\cite[Corollary~4.19]{Santos-progress}}]
\label{lem:containment_z}
Let $C$ be a flag normal complex, $x$ and $y$ be two consecutive anchors, with $x$ coming first, along a monotone conservative path~$\Gamma=[F_0,\dots,F_N]$.
If $z$ is a neighbor of $y$ that belongs to a facet $F_i$ anchored at $x$, then $z$ belongs to all facets between $F_i$ and the first one anchored at $y$.
\end{lemma}

\begin{proof}
The proof goes by induction on the dimension of~$C$ and is obvious in dimension~$1$. Now in higher dimension, we can assume without loss of generality that $x$ and $y$ are the only two anchors in the path, that $i=0$, that $y$ first appears in $F_N$, and that $x\ne z$.
Consider the monotone and conservative path $\Gamma'$ obtained from $\Gamma$ in $\lk_{C}(x)$.
Since $y$ is an anchor of $\Gamma$, it is in the target set of $\Gamma'$ and $\Gamma'$ finishes at a facet containing~$y$.
Since~$\{x,y\},\{x,z\}$ and~$\{y,z\}$ are edges of~$C$, which is flag, the set $\{x,y,z\}$ is a face of $C$ so that~$\{x,z\}$ and~$\{y,z\}$ are in~$\lk_{C}(x)$.
The vertex $z$ is in the first facet of $\Gamma'$ and at distance 1 of the target set. So either it is the anchor of~$\Gamma'$ and we are done, or the first anchor of~$\Gamma'$ then has to be at distance~$1$ of~$y$ to give an admissible order. Call this anchor $x'$.
Then, $\Gamma'$ satisfies all the conditions in the statement, with respect to the vertices $x'$, $y$ and $z$. 
Since flagness is preserved under taking links, we can apply the induction hypothesis in~$\lk_{C}(x)$. Thus, $z$ is in all facets of the path~$\Gamma'$, and so of the path~$\Gamma$.
\end{proof}

\begin{corollary}[{\cite[Section~3]{AdiprasitoBenedetti}, see also~\cite[Corollary~4.20]{Santos-progress}}]
\label{coro:hirsch-for-flag}
Every monotone and conservative path in a flag normal complex is a nonrevisiting path. In particular, its length is bounded by the Hirsch bound.
\end{corollary}

\begin{proof}
Let $\Gamma=[F_0,\dots,F_N]$ be a monotone conservative path towards a target set~$S$ in a flag normal complex $C$.
The proof is by a double induction on the dimension of $C$ and the length $N$ of the path $\Gamma$.
For dimension $d=1$ or for $N=1$, the result is clear.

Now, assume $d>1$ and $N>1$.
Let $x$ be the first anchor of $\Gamma$ and $y$ be the second anchor, if any. We call~$\Gamma'$ the path~$[F_0\ssm\{x\},\dots,F_k\ssm\{x\}]$ where~$[F_0,\dots,F_{k-1}]$ is the part of~$\Gamma$ anchored at~$x$.
By induction on the dimension, the monotone and conservative path $\Gamma'$ in $\lk_{C}(x)$ is nonrevisiting.
By induction on~$N$, the tail~$\Gamma_2\eqdef[F_k,\dots,F_N]$ of $\Gamma$ once $x$ is not an anchor is nonrevisiting.
So it only remains to show that, if there is a vertex $z$ used in $\Gamma'$ and~$\Gamma_2$, then it has to be contained in the last facet of $\Gamma'$, where a change of anchor occurs.  Since~$z$ is a vertex in~$\lk_C(x)$, it is also a neighbor of~$x$ in~$C$. Moreover~$x$ and~$y$ are consecutive anchors of~$\Gamma$ and so are neighbors in~$C$. Set~$\ell=\vdist_C(y,S)$ and suppose, for the sake of contradiction, that~$z$ is not a neighbor of~$y$. Since~$y$ is the anchor of~$\Gamma$ following~$x$, we have~$\vdist_C(x,S)=\ell+1$, and so~$\vdist_C(z,S)\ge\ell+1$ because of admissibility. Since~$z$ appears in a facet of~$\Gamma_2$, the anchor of this facet has to be at vertex-distance at most~$\ell-1$ from~$S$. Indeed it cannot be~$y$ and the sequence of anchors forms a shortest vertex path to~$S$.
Since~$z$ is a neighbor of this anchor in~$C$, we have~$\vdist_C(z,S)\le\ell$; a contradiction.
Hence~$y$ and~$z$ are neighbors in~$C$; and Lemma~\ref{lem:containment_z} applies to $C$ and $x,y$ and $z$.
Thus~$z$ belongs to the last facet of~$\Gamma'$; which concludes.
\end{proof}

\begin{remark}
Both the monotonicity and conservativeness assumptions are necessary in Corollary~\ref{coro:hirsch-for-flag}. 
On the one hand, consider the nonmonotone but conservative path obtained in Example~\ref{ex:non_ex} in the complex $C=\left\{\{1,2\},\{1,3\},\{1,4\}\right\}$.
The complex $C$ is normal and flag.
Let $F=(1,2)$, $G=(1,3)$, and $S=\{4\}$, the path $[F,G,F]$ is nonmonotone, conservative and revisiting.
On the other hand, consider the complex presented in Figure~\ref{fig:mono_revisiting} and the shown monotone but not conservative path which is revisiting.
Again, the complex is flag and normal.
\end{remark}

\begin{figure}[!hbtp]
 \centerline{\begin{tikzpicture}[vertex/.style={inner sep=1pt,circle,draw=blue!75!black,fill=black!40,thick}]

\def\rad{3}

\node[vertex] (s1) at (0,0) {$s_1$};
\node[vertex] (s2) at (\rad,0) {$s_2$} edge[very thick] (s1);
\node[vertex] (s3) at (\rad/2,-\rad/2) {$s_3$} edge[very thick] (s1) edge[very thick] (s2);

\node[vertex] (f1) at (\rad/2,\rad/2) {$f_1$} edge[very thick] (s1) edge[very thick] (s2);

\node[vertex] (d1) at (\rad,3*\rad/20) {} edge[very thick] (s2) edge[very thick] (f1);
\node[vertex] (d2) at (\rad,6*\rad/20) {} edge[very thick] (d1) edge[very thick] (f1);
\node[vertex] (d3) at (\rad,9*\rad/20) {} edge[very thick] (d2) edge[very thick] (f1);
\node[vertex] (d4) at (\rad,12*\rad/20) {} edge[very thick] (d3) edge[very thick] (f1);
\node[vertex] (d5) at (\rad,15*\rad/20) {} edge[very thick] (d4) edge[very thick] (f1);
\node[vertex] (d6) at (\rad,18*\rad/20) {} edge[very thick] (d5) edge[very thick] (f1);
\node[vertex] (d7) at (\rad,21*\rad/20) {} edge[very thick] (d6) edge[very thick] (f1);
\node[vertex] (d8) at (\rad,24*\rad/20) {} edge[very thick] (d7) edge[very thick] (f1);
\node[vertex] (d9) at (\rad,27*\rad/20) {} edge[very thick] (d8) edge[very thick] (f1);

\node[vertex] (f3) at (\rad,3*\rad/2) {$f_3$} edge[very thick] (d9) edge[very thick] (f1);

\node[vertex] (g1) at (0,25*\rad/40) {$v_3$} edge[very thick] (s1) edge[very thick] (f1);
\node[vertex] (g2) at (0,30*\rad/40) {} edge[very thick] (g1) edge[very thick] (f1);
\node[vertex] (g3) at (0,35*\rad/40) {} edge[very thick] (g2) edge[very thick] (f1);
\node[vertex] (g4) at (0,40*\rad/40) {} edge[very thick] (g3) edge[very thick] (f1);
\node[vertex] (g5) at (0,45*\rad/40) {} edge[very thick] (g4) edge[very thick] (f1);

\node[vertex] (f2) at (0,5*\rad/4) {$f_2$} edge[very thick] (f3) edge[very thick] (f1) edge[very thick] (g5);

\node[vertex] (v2) at (-\rad/4,3*\rad/4) {$v_2$} edge[very thick] (s1) 
                                                edge[very thick] (g1)
					        edge[very thick] (g2)
					        edge[very thick] (g3) 
					        edge[very thick] (g4) 
					        edge[very thick] (g5) 
					        edge[very thick] (f2);

\node[vertex] (v1) at (-\rad,3*\rad/2) {$v_1$} edge[very thick] (f3) edge[very thick] (f2) edge[very thick] (v2) edge[very thick] (s1);

\draw[thick,dashed, rounded corners,->] ($(f3)!0.5!(f2)!0.2!(f1)$) -- ($(f3)!0.5!(f2)!0.33!(v1)$) -- ($(v1)!0.5!(v2)!0.2!(f2)$) -- ($(v1)!0.5!(s1)!0.33!(v2)$) -- ($(s1)!0.5!(g1)$) -- ($(s1)!0.5!(f1)!0.33!(s2)$) -- ($(s1)!0.5!(s2)!0.2!(s3)$);

\node at ($(f1)!0.5!(f2)!0.33!(f3)$) {$F_0$};
\node at ($(s1)!0.5!(s2)!0.33!(s3)$) {$S$};

\node at (2.25*\rad,21*\rad/14) {$F_0=(f_1,f_2,f_3)\quad \Lambda_{F_0}=(1,6,1)$};
\node at (2.25*\rad,19*\rad/14) {$F_1=(v_1,f_2,f_3)\quad \Lambda_{F_1}=(1,2,1)$};
\node at (2.25*\rad,17*\rad/14) {$F_2=(v_1,v_2,f_2)\quad \Lambda_{F_2}=(1,1,1)$};
\node at (2.25*\rad,15*\rad/14) {$F_3=(s_1,v_2,v_1)\quad \Lambda_{F_3}=(0,3,1)$};
\node at (2.25*\rad,13*\rad/14) {$F_4=(s_1,v_3,v_2)\quad \Lambda_{F_4}=(0,2,1)$};
\node at (2.25*\rad,11*\rad/14) {$F_5=(s_1,f_1,v_3)\quad \Lambda_{F_5}=(0,1,1)$};
\node at (2.25*\rad,9*\rad/14) {$F_6=(s_1,s_2,f_1)\quad \Lambda_{F_5}=(0,0,1)$};

\end{tikzpicture}}
 \caption{A flag normal complex with a monotone nonconservative path which is revisiting. The step $[F_0,F_1]$ is not conservative.}
 \label{fig:mono_revisiting}
\end{figure}

Even in flag normal complexes, monotone  conservative paths do not yield, or even approximate, shortest paths:

\begin{lemma}
\label{lemma:dim2examples}
There are flag 2-balls (and flag 3-polytopes) with the following behaviors: 
 \begin{enumerate}[(a)]
  \item the difference in length in monotone conservative paths between two given facets can be arbitrarily large. 
  \item no monotone conservative path is a shortest path between two given facets.
 \end{enumerate}
\end{lemma}

\begin{proof}
In Figure~\ref{fig:large_diff}, we see two examples of flag 2-balls with the needed properties.
In the one on the left, the first anchor can be $f_1$ or $f_2$.
If we choose $f_1$, the path to the facet $S$ is going to be short.
Otherwise, choosing $f_2$ may lead to an arbitrary large path depending on the number of vertices involved in the ``comb'' region.
In the complex on the right, the first anchor has to be $f_2$ and the path can be very long for the same reason.
The shortest path from $F$ to $S$ can thus be arbitrarily shorter than any monotone conservative path.

It is not difficult to make these examples as flag 3-polytopes by adding a vertex $v$ joined to the boundary and making edge-subdivisions of the edge $\{v,f_3\}$ sufficiently many times. Adding~$v$ and the edges to the boundary gives a flag 2-sphere since there are no empty triangles. Finally, flagness and polytopality is preserved by doing edge-subdivisions.

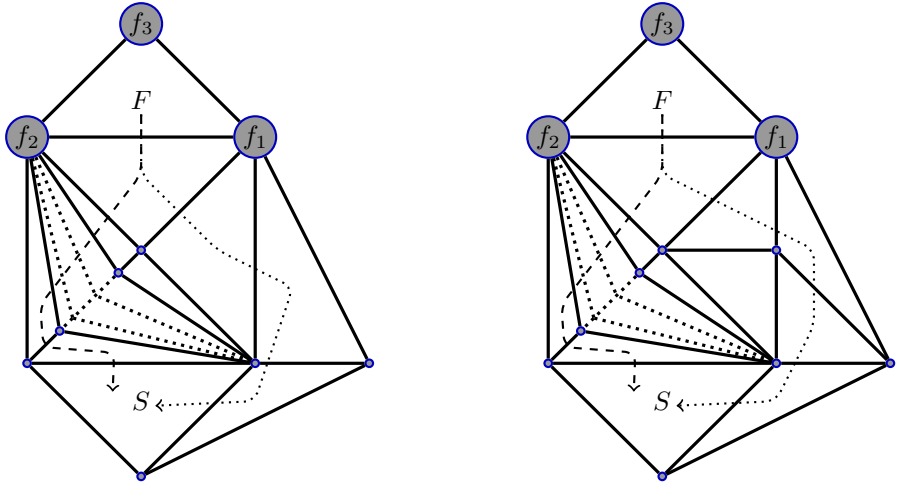
\begin{figure}[!hbtp]
 \centerline{\begin{tabular}{c@{\hspace{2cm}}c}

\begin{tikzpicture}[vertex/.style={inner sep=1pt,circle,draw=blue!75!black,fill=black!40,thick}]

\def\rad{3}

\node[vertex] (s1) at (0,0) {};
\node[vertex] (s2) at (\rad,0) {} edge[very thick] (s1) edge[very thick, dotted] (\rad/5,\rad/5) edge[very thick, dotted] (3*\rad/10,3*\rad/10);
\node[vertex] (s3) at (\rad/2,-\rad/2) {} edge[very thick] (s1) edge[very thick] (s2);

\node[vertex] (d1) at (\rad/7,\rad/7) {} edge[very thick] (s1) edge[very thick] (s2);
\node[vertex] (d2) at (2*\rad/5,2*\rad/5) {} edge[very thick,dotted] (d1) edge[very thick] (s2);

\node[vertex] (c) at (\rad/2,\rad/2) {} edge[very thick] (d2) edge[very thick] (s2);

\node[vertex] (f2) at (0,\rad) {$f_2$} edge[very thick] (s1) edge[very thick] (c) edge[very thick] (d1) edge[very thick] (d2) edge[very thick, dotted] (\rad/5,\rad/5) edge[very thick, dotted] (3*\rad/10,3*\rad/10);

\node[vertex] (f1) at (\rad,\rad) {$f_1$} edge[very thick] (f2) edge[very thick] (c) edge[very thick] (s2);

\node[vertex] (f3) at (\rad/2,3*\rad/2) {$f_3$} edge[very thick] (f1) edge[very thick] (f2);

\node[vertex] (p) at (3*\rad/2,0) {} edge[very thick] (f1) edge[very thick] (s2) edge[very thick] (s3);

\draw[thick,dotted, rounded corners,->] ($(f1)!0.5!(f2)!0.2!(f3)$) -- ($(f1)!0.5!(f2)!0.33!(c)$) -- ($(f1)!0.5!(s2)!0.33!(c)$) -- ($(f1)!0.5!(s2)!0.33!(p)$) -- ($(s3)!0.5!(s2)!0.33!(p)$) -- ($(s3)!0.5!(s2)!0.25!(s1)$);

\draw[thick,dashed, rounded corners,->] ($(f1)!0.5!(f2)!0.2!(f3)$) -- ($(f1)!0.5!(f2)!0.33!(c)$) -- ($(d1)!0.5!(s1)!0.2!(f2)$) -- ($(s1)!0.5!(d1)$) -- ($(s2)!0.5!(s1)!0.33!(d1)$) -- ($(s3)!0.5!(s2)!0.5!(s1)$);

\node at ($(f1)!0.5!(f2)!0.33!(f3)$) {$F$};
\node at ($(s1)!0.5!(s2)!0.33!(s3)$) {$S$};

\end{tikzpicture}

&

\begin{tikzpicture}[vertex/.style={inner sep=1pt,circle,draw=blue!75!black,fill=black!40,thick}]

\def\rad{3}

\node[vertex] (s1) at (0,0) {};
\node[vertex] (s2) at (\rad,0) {} edge[very thick] (s1) edge[very thick, dotted] (\rad/5,\rad/5) edge[very thick, dotted] (3*\rad/10,3*\rad/10);
\node[vertex] (s3) at (\rad/2,-\rad/2) {} edge[very thick] (s1) edge[very thick] (s2);

\node[vertex] (d1) at (\rad/7,\rad/7) {} edge[very thick] (s1) edge[very thick] (s2);
\node[vertex] (d2) at (2*\rad/5,2*\rad/5) {} edge[very thick,dotted] (d1) edge[very thick] (s2);

\node[vertex] (c) at (\rad/2,\rad/2) {} edge[very thick] (d2) edge[very thick] (s2);

\node[vertex] (e) at (\rad,\rad/2) {} edge[very thick] (s2) edge[very thick] (c);

\node[vertex] (f2) at (0,\rad) {$f_2$} edge[very thick] (s1) edge[very thick] (c) edge[very thick] (d1) edge[very thick] (d2) edge[very thick, dotted] (\rad/5,\rad/5) edge[very thick, dotted] (3*\rad/10,3*\rad/10);

\node[vertex] (f1) at (\rad,\rad) {$f_1$} edge[very thick] (f2) edge[very thick] (c) edge[very thick] (e);

\node[vertex] (f3) at (\rad/2,3*\rad/2) {$f_3$} edge[very thick] (f1) edge[very thick] (f2);

\node[vertex] (p) at (3*\rad/2,0) {} edge[very thick] (f1) edge[very thick] (s2) edge[very thick] (s3) edge[very thick] (e);

\draw[thick,dotted, rounded corners,->] ($(f1)!0.5!(f2)!0.2!(f3)$) -- ($(f1)!0.5!(f2)!0.33!(c)$) -- ($(f1)!0.5!(e)!0.33!(c)$) -- ($(f1)!0.5!(e)!0.33!(p)$) -- ($(e)!0.5!(s2)!0.33!(p)$) -- ($(s3)!0.5!(s2)!0.33!(p)$) -- ($(s3)!0.5!(s2)!0.25!(s1)$);

\draw[thick,dashed, rounded corners,->] ($(f1)!0.5!(f2)!0.2!(f3)$) -- ($(f1)!0.5!(f2)!0.33!(c)$) -- ($(d1)!0.5!(s1)!0.2!(f2)$) -- ($(s1)!0.5!(d1)$) -- ($(s2)!0.5!(s1)!0.33!(d1)$) -- ($(s3)!0.5!(s2)!0.5!(s1)$);

\node at ($(f1)!0.5!(f2)!0.33!(f3)$) {$F$};
\node at ($(s1)!0.5!(s2)!0.33!(s3)$) {$S$};

\end{tikzpicture}

\end{tabular}}
 \caption{The flag 2-ball on the left has two monotone conservative paths from $F$ to $S$ whose difference in length is large. The flag 2-ball on the right has a monotone conservative path from $F$ to $S$ much longer than the shortest path.}
 \label{fig:large_diff}
\end{figure}
\end{proof}

%%%%%%% A linear bound in fixed dimension %%%%%%%%%
%%%%%%%%%%%%%%%%%%%%%%%%%%%%%%%%%%%%%%%%%%%%%%%%%%%
\subsection{A linear bound in fixed dimension}
\label{subsec:Larman}

We show here that monotone conservative paths lead to the classical bound of Larman~\cite{larman_paths_1970}, Barnette~\cite{barnette_upperbound_1974}, and Eisenbrand et al.~\cite{eisenbrand_diameter_2010}.

\begin{theorem}[\cite{larman_paths_1970,barnette_upperbound_1974,eisenbrand_diameter_2010}]
\label{thm:BarnetteLarman}
The length of a monotone conservative path in a~$(d-1)$-dimensional pure complex~$C$ on~$n$ vertices ($d\ge2$) is at most~$n2^{d-2}$. In particular if~$C$ is normal, the same inequality holds for its diameter. Moreover, if~$C$ is a pseudomanifold without boundary, the same statement holds with the bound~$n2^{d-3}$.
\end{theorem}

The statement is meant with respect to any target set. In particular we never make the target set explicit in the proof for brevity reasons.

As mentioned in the introduction the bound~$n2^{d-2}$ is, modulo a (small) constant, one of the best that are proved for the diameter of normal simplicial complexes, or even for the more restricted case of polytopal simplicial complexes.

\begin{proof}
The proof is by induction on $d$.
In the case $d=1$, either $n=1$, and then $N=0 \le n 2^{d-2} =1/2$, or
$n\ge 2$, and then $N\le 1= 2\cdot\frac{1}{2}\leq n 2^{d-2}$.

For $d\ge 2$ let $x_1,\dots,x_k$ be the sequence of anchors along $\Gamma$. Let $\Gamma_i=[F^i_0,\dots,F^i_{N_i}]$ be the subpath of $\Gamma$ anchored at $x_i$, including the facet at which the change of anchor takes places (that is, $F^{i-1}_{N_{i}}=F^i_0$ is the first facet anchored at $x_i$, obtained in a step that was still anchored at $x_{i-1}$). 

The path $\Gamma_i$ is monotone and conservative in the star of $x_i$. Deleting $x_i$ from it gives a monotone conservative path in $\lk_C(x_i)$ which, by inductive hypothesis, has length at most
\[
n_i 2^{d-3},
\]
where $n_i$ is the number of vertices other than $x_i$ used in~$\Gamma_i$.
Now, since the distances $\vdist(x_i,S)$ are decreasing, no vertex of $C$ can be in more than two links of the form $\lk_{C}(x_i)$. This shows that $\sum_{i=1}^kn_i\leq 2n$. Using this fact and the induction hypothesis, we have
\[
N \quad=\quad \sum_{i=1}^k N_i \quad\le\quad \sum_{i = 1}^k n_i2^{d-3} \quad\leq\quad 2^{d-2}n.
\]

For pseudomanifolds the proof follows the same ideas, but the induction starts with $d=2$. A $1$-dimensional pseudomanifold without boundary is a cycle on~$n$ vertices, and the length of monotone conservative paths in it is indeed bounded by $n/2$.

Now the normality hypothesis ensures the existence of monotone conservative paths between any two facets, which concludes for the bound on the diameter in this case.
\end{proof}

%%%%%%% Refined bound for banner complexes %%%%%%%%%
%%%%%%%%%%%%%%%%%%%%%%%%%%%%%%%%%%%%%%%%%%%%%%%%%%%%
\subsection{Refined bound for banner complexes}
\label{subsec:BannerComplexes}

We now look at monotone conservative paths in banner complexes.
We begin by recalling the definition by Klee and Novik~\cite{KleeNovik}.
Then, we prove upper bounds for the length of monotone conservative paths for these complexes, and therefore on their diameter.

\begin{definition}
Let~$C$ be a pure simplicial complex of dimension~$(d-1)$. A~\defn{critical clique} is a set~$T$ of vertices of~$C$ forming a clique in the~$1$-skeleton of~$C$ and such that there exists a vertex~$v\in T$ such that~$T\ssm \{v\}$ is a face of~$C$.
For~$k\in\{2, \dots,d-1\}$, the complex~$C$ is said to be~\defn{$k$-banner} if every critical clique of size at least~$(k+1)$ is a face of~$C$.
\end{definition}

Notice that any~$(d-1)$-dimensional complex has no face with~$(d+1)$ vertices, and thus cannot contain a critical clique of size~$(d+2)$. So any~$(d-1)$-dimensional complex is~$(d+1)$-banner.
Other properties on banner complexes are to be found in the article~\cite{KleeNovik}.
Among them is the fact that any~$k$-banner complex is also~$(k+1)$-banner, and that flag~and~$2$-banner complexes coincide.
Thus, banner complexes form a nested sequence of families of complexes starting with flag complexes for $k=2$ to finish with all complexes for $k=d+1$.
Let us also mention that the original definition allows for $k=1$, but since $1$-banner and $2$-banner complexes are exactly the same, we prefer not to consider $1$-banner complexes.

\begin{remark}
Every minimal nonface of size at least~$3$ in a complex is also a critical clique. In particular, if~$C$ is~$k$-banner then every minimal nonface of it has size at most~$k$. In other words, the Stanley--Reisner ring of a $k$-banner complex is generated in degree at most~$k$.  The converse may not hold.
\end{remark}

We also recall the following result on links in banner complexes.

\begin{lemma}[{\cite[Lemma~3.4]{KleeNovik}}]
\label{lem:linkBanner}
Let $k\ge2$, $C$ be a simplicial complex and $x$ be a vertex of $C$. If~$C$ is a~$k$-banner, then the link~$\lk_C(x)$ of~$x$ in~$C$ is~$(k-1)$-banner.
\end{lemma}

\begin{proof}
Any critical clique in~$\lk_C(x)$ is, after adding~$x$, a critical clique of~$C$ of size one more.
\end{proof}

This lemma allows us to provide an upper bound for the length of monotone conservative paths in banner complexes, combining the proofs of Theorem~\ref{thm:BarnetteLarman} and Corollary~\ref{coro:hirsch-for-flag}, which are special cases of the following:

\begin{theorem}
\label{thm:upperBound}
Let $d\ge2$, $k\ge2$, and $C$ be a normal $k$-banner $(d-1)$-complex on~$n$ vertices.
If $\Gamma=[F_0,\dots,F_{\ell}]$ is a monotone conservative path in $C$ towards a  set~$S$, 
then the length of~$\Gamma$ satisfies~$\ell\le n2^{k-2}$. 
\end{theorem}

\begin{proof}
We show the result by induction on the dimension~$d-1$. For~$d=2$, monotone conservative paths are the shortest paths in~$C$ so that~$\ell\le n$. The inequality thus holds in this case.
Suppose now that~$d>2$. If~$k=2$, then we know that~$C$ is flag, implying the Hirsch bound~\cite{AdiprasitoBenedetti}, and the inequality holds.

 So suppose now that~$k\ge3$ and consider the dual paths~$\Gamma_1,\dots,\Gamma_r$ obtained from splitting~$\Gamma$ according to its sequence of anchors, that is~$\Gamma_i$ is the greatest subpath of~$\Gamma$ in which the anchor of any facet is the~$i^{\mbox{\tiny th}}$ anchor~$x_i$ of the sequence of anchors of~$\Gamma$, for~$i\in[r]$.
 In particular all~$\Gamma_i$'s are monotone conservative paths from their starting facet towards the set~$S$.
We fix~$i\in[r]$ and consider the link~$\Gamma'_i$ of~$x_i$ in~$\Gamma_i$. It is a monotone conservative path towards the set~$S'$ of Definition~\ref{defi:vector-of-distances}. Moreover it is a monotone conservative path in the link~$\lk_C(x_i)$ of~$x_i$ in~$C$.

Since~$C$ is~$k$-banner, Lemma~\ref{lem:linkBanner} ensures that this link is~$(k-1)$-banner. Since~$k\ge3$, we can thus apply the induction hypothesis to~$\Gamma'_i$.
Let~$n_i$ be the number of vertices of~$C$ that appear in at least one facet of~$\Gamma'_i$, the induction hypothesis says that~$\Gamma'_i$, and thus~$\Gamma_i$, has length at most~$n_i2^{k-3}$.
So we have
\[
\ell \le \sum_{i=1}^{r} n_i2^{k-3} = 2^{k-3}\sum_{i=1}^{r} n_i.
\]
Now recall that each facet of the path~$\Gamma_i$ contains~$x_i$ and vertices that are either at distance $\vdist(x_i,S)$ or $\vdist(x_i,S)~+~1$ of the set~$S$.
Since~$\vdist(x_i,S) = \vdist(x_{i + 1}, S) + 1$, a vertex of~$C$ cannot appear in more than two consecutive~$\Gamma_i$'s.
So the sum of the~$n_i$'s is smaller than~$2n$, which concludes.
\end{proof}

Observe that Theorem~\ref{thm:upperBound} somehow interpolates between Theorem~\ref{thm:BarnetteLarman} and Corollary~\ref{coro:hirsch-for-flag}:
\begin{itemize}
\item Since every $(d-1)$-complex is $(d+1)$-banner, substituting $k$ by $d+1$, we recover (except for a factor of two) the bounds in Theorem~\ref{thm:BarnetteLarman}.
\item Since flag complexes are (the same as) $2$-banner complexes, substituting $k$ by $2$, we recover (almost) the Hirsch bound of Corollary~\ref{coro:hirsch-for-flag}.
\end{itemize}

%%%%%%%%%%%%%%%%%%%%%%%%%%%%%%%%%%%%%%%%%%%%%%%%%%%%%%%%%%%%%%%%%%%%%%%
%%%%%%% Monotone conservative paths can be exponentially long %%%%%%%%%
%%%%%%%%%%%%%%%%%%%%%%%%%%%%%%%%%%%%%%%%%%%%%%%%%%%%%%%%%%%%%%%%%%%%%%% 

\section{Monotone conservative paths can be exponentially long}
\label{sec:lowerbound}

In this final section, we present two constructions of normal complexes that contain exponentially long monotone conservative paths.
The first one is a vertex-decomposable ball and the second, a simplicial polytope whose boundary complex is vertex-decomposable.
Both construction rely on lemmas (Lemma~\ref{lemma:exampleBall} and~\ref{lemma:exampleSphere}) that provide the inductive procedure of the construction.

%%%%%%% A simplicial ball with exponential monotone conservative paths %%%%%%%%%
%%%%%%%%%%%%%%%%%%%%%%%%%%%%%%%%%%%%%%%%%%%%%%%%%%%%%%%%%%%%%%%%%%%%%%%%%%%%%%%%
\subsection{A simplicial ball with exponential monotone conservative paths}

\begin{lemma}
\label{lemma:exampleBall}
Let $\ball$ be a $(d-1)$-dimensional simplicial ball with $n$ vertices. Suppose there are facets $F_1$ and $F_2$ and vertices $x_1$ and $x_2$ satisfying the following properties:
\begin{enumerate}
\item $F_i$ is the only facet containing $x_i$, for $i=1,2$.
\item Every monotone conservative path from $F_1$ to $F_2$ or from $F_2$ to $F_1$ has length at least $L$.
\end{enumerate}
Then, for every $k\ge 2d$ there is a $d$-dimensional simplicial ball $\ball'$ with $n+k+1$ vertices having facets $F'_1$ and $F'_2$ and vertices $x'_1$ and $x'_2$ satisfying:
\begin{enumerate}
\item $F'_i$ is the only facet containing $x'_i$, for $i=1,2$.
\item Every monotone conservative path from $F'_1$ to $F'_2$ or from $F'_2$ to $F'_1$ has length at least $2L+k$.
\end{enumerate}
\end{lemma}

Observe that the statement does not specify the target sets for the monotone conservative paths. The expression ``for every monotone conservative path'' is then meant whatever the target set $S$ may be; as long as that target set $S$ allows for the existence of the claimed monotone conservative paths; in the case of the lemma this is equivalent to $x_2\subset S \subset F_2$. 

\begin{proof}
The condition on the vertex $x_i$ implies that the facet $F_i$ has a codimension-one face $G_i$ contained in the boundary of $\ball$.
Consider the one-point-suspension $\ops_\ball(x_1)$ of $\ball$ on the vertex~$x_1$, and call $u_1$ and $u_2$ the suspension vertices.
Observe that $F_2 \cup u_i$ is a facet containing the codimension-one face $G_2\cup u_i$, for $i=1,2$. 

Let $C_\ell$ be the following $d$-complex on $\ell+d$ vertices $\{u,v_1,v_2,\dots, v_{\ell+d-1}\}$ and $\ell$ facets:
\[
C_\ell:= \{ \{u, v_i,\dots, v_{i+d-1}\} : i=1,\dots \ell\}.
\]
Observe that for $\ell\ge d$ the codimension-one boundary face $G:=\{u,v_1,\dots, v_{d-1}\}$ and the facet $F:=\{u,v_\ell,\dots, v_{\ell+d-1}\}$ only have $u$ in common.

Let now $k_1,k_2\in \N$ be such that $k_1+k_2=k$ and $k_i\ge d$, $i=1,2$.
Our complex $\ball'$ consists of the one-point-suspension $\ops_\ball(x_1)$ of $\ball$ with a copy of $C_{k_1}$ glued by identifying $G_2\cup u_1$ with $G\cup u$
and  a copy of $C_{k_2}$ glued by identifying $G_2\cup u_2$ with $G\cup u$. The facets $F'_1$ and $F'_2$ are the $F$'s in $C_{k_1}$ and $C_{k_2}$, and the vertices $x'_i$ are the last vertices $v_{k_1+d-1}$ and $v_{k_2+d-1}$ in them. Let us see that the stated properties are satisfied. We concentrate on paths from $F'_1$ to $F'_2$ but all the assertions are valid for the reverse paths, by symmetry.

\begin{itemize}
\item The number of vertices in $\ball'$ equals the $n+1$ in $\ops_\ball(x_1)$ plus the $k+2d$ vertices in $C_{k_1}$ plus $C_{k_2}$, minus the $2d$ vertices along which we glue.

\item The facets $F'_i$ and vertices $x'_i$ clearly satisfy property 1 in the statement.

\item Since $\ball'$ is a ball, it is normal and, by Lemma~\ref{coro:path-existence-2}, there is at least one monotone conservative path from $F'_1$ to $F'_2$.

\item Every facet path (monotone and conservative or not)  from $F'_1$ to $F'_2$ starts by going through the $k_1$ facets of $C_{k_1}$ and finishes by going through the $k_2$ facets $C_{k_2}$. Apart of these $k$ steps, the path consists of a path between $F_2 \cup u_1$ and $F_2 \cup u_2$ inside  $\ops_\ball(x_1)$.

\item The unique closest vertices between $F'_1$ and $F'_2$ are $u_1$ and $u_2$, which form an edge. In particular, every monotone conservative path from $F'_1$ to $F'_2$ (or vice versa) has two anchors, $u_1$ and $u_2$.

\item The part of the path anchored at $u_1$ goes from $F_2 \cup u_1$ to $\ops_{F_1}(x_1)=F_1\ssm\{x_1\}\cup \{u_1,u_2\}$. Its link at $u_1$ is hence a monotone conservative path from $F_2$ to $F_1$ in $\ball$, so it has length at least $L$.

\item The part of the path anchored at $u_2$ goes from $\ops_{F_1}(x_1)=F_1\ssm\{x_1\}\cup \{u_1,u_2\}$ to $F_2 \cup u_2$. Its link at $u_2$ is hence a monotone conservative path from $F_1$ to $F_2$ in $\ball$, so it has length at least $L$.\qedhere
\end{itemize}
\end{proof}

\begin{theorem}
\label{thm:exampleBall}
For every $d\ge 2$ and every $N\ge 4$ there is a $(d-1)$-ball $\ball_d$ with $N+d^2$ vertices and two facets $F_1$ and $F_2$ in it such that every  monotone conservative path between them has length at least $2^{d-2} (N+3)$.
\end{theorem}

\begin{proof}
Let $\ball_2$ be a path with $N+4$ vertices and iterate the process of Lemma~\ref{lemma:exampleBall} with~$k=2d$.
\end{proof}

\begin{remark}
It is easy to check that the construction of Lemma~\ref{lemma:exampleBall} preserves shellability and vertex-decomposability. Moreover, if every monotone conservative path in $\ball$ is a Hamiltonian path in the dual graph, then the same happens in $\ball'$. In particular, in the statement of Theorem~\ref{thm:exampleBall} we can add that $\ball_d$ is vertex-decomposable and that every monotone conservative path from $F_1$ to $F_2$ is Hamiltonian.
\end{remark}

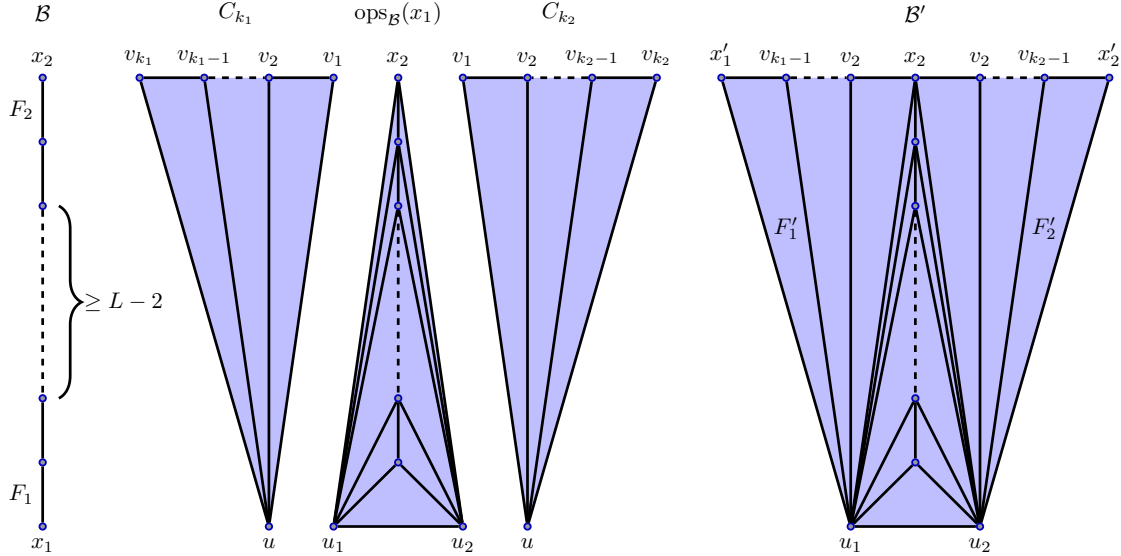
\begin{figure}[!htbp]
\begin{center}
 \resizebox{\hsize}{!}{\begin{tikzpicture}[vertex/.style={inner sep=1pt,circle,draw=blue!75!black,fill=black!40,thick,rotate=90}]

\def\stepone{5.5}
\def\steptwo{0}
\def\stepthree{-8}

\begin{scope}[rotate=90]
\node at (+8,\stepone) {$\mathcal{B}$};

\node[vertex,label=left:{$x_1$}] (x1_1) at (0,\stepone) {};
\node[vertex] (a1_1) at (1,\stepone) {} edge[very thick] node[left, midway] {$F_1$} (x1_1);
\node[vertex] (a2_1) at (2,\stepone) {} edge[very thick] (a1_1);
\node[vertex] (al2_1) at (5,\stepone) {} edge[very thick,dashed] (a2_1);
\node[vertex] (al1_1) at (6,\stepone) {} edge[very thick] (al2_1);
\node[vertex,label=right:{$x_2$}] (x2_1) at (7,\stepone) {} edge[very thick] node[left, midway] {$F_2$} (al1_1);

\draw[very thick,decorate,decoration={brace,amplitude=10pt},xshift=0cm,yshift=-0.25cm] (5,\stepone) -- (2,\stepone) node [black,midway,xshift=1cm] {$\geq L-2$};

\node at (+8,\steptwo) {$\ops_{\mathcal{B}}(x_1)$};
\node at (+8,\steptwo+2.5) {$C_{k_{1}}$};
\node at (+8,\steptwo-2.5) {$C_{k_{2}}$};

\coordinate (cu1_2) at (0,\steptwo+1);
\coordinate (cu2_2) at (0,\steptwo-1);
\coordinate (ca1_2) at (1,\steptwo);
\coordinate (ca2_2) at (2,\steptwo);
\coordinate (cal2_2) at (5,\steptwo);
\coordinate (cal1_2) at (6,\steptwo);
\coordinate (cx2_2) at (7,\steptwo);
\coordinate (cv1_2) at (7,\steptwo+1);
\coordinate (cv2_2) at (7,\steptwo+2);
\coordinate (cvk11_2) at (7,\steptwo+3);
\coordinate (cvk1_2) at (7,\steptwo+4);
\coordinate (cvv1_2) at (7,\steptwo-1);
\coordinate (cvv2_2) at (7,\steptwo-2);
\coordinate (cvvk11_2) at (7,\steptwo-3);
\coordinate (cvvk1_2) at (7,\steptwo-4);

\coordinate (cgu_2) at (0,\steptwo+2);
\coordinate (chu_2) at (0,\steptwo-2);

\fill[blue!25!white] (cu1_2) -- (cu2_2) -- (cx2_2) -- cycle;
\fill[blue!25!white] (cv1_2) -- (cvk1_2) -- (cgu_2) -- cycle;
\fill[blue!25!white] (cvv1_2) -- (cvvk1_2) -- (chu_2) -- cycle;

\node[vertex,label=left:{$u_1$}] (u1_2) at (cu1_2) {};
\node[vertex,label=left:{$u_2$}] (u2_2) at (cu2_2) {} edge[very thick] (u1_2);
\node[vertex] (a1_2) at (ca1_2) {} edge[very thick] (u1_2) edge[very thick] (u2_2);
\node[vertex] (a2_2) at (ca2_2) {} edge[very thick] (a1_2) edge[very thick] (u1_2) edge[very thick] (u2_2);
\node[vertex] (al2_2) at (cal2_2) {} edge[very thick,dashed] (a2_2) edge[very thick] (u1_2) edge[very thick] (u2_2);
\node[vertex] (al1_2) at (cal1_2) {} edge[very thick] (al2_2) edge[very thick] (u1_2) edge[very thick] (u2_2);
\node[vertex,label=right:{$x_2$}] (x2_2) at (cx2_2) {} edge[very thick] (al1_2) edge[very thick] (u1_2) edge[very thick] (u2_2);

\node[vertex,label=right:{$v_1$}] (v1_2) at (cv1_2) {};
\node[vertex,label=left:{$u$}] (gu_2) at (cgu_2) {} edge[very thick] (v1_2);
\node[vertex,label=right:{$v_2$}] (v2_2) at (cv2_2) {} edge[very thick] (v1_2) edge[very thick] (gu_2);
\node[vertex,label=right:{$v_{k_{1}-1}$}] (vk11_2) at (cvk11_2) {} edge[very thick,dashed] (v2_2) edge[very thick] (gu_2);
\node[vertex,label=right:{$v_{k_{1}}$}] (vk1_2) at (cvk1_2) {} edge[very thick] (vk11_2) edge[very thick] (gu_2);

\node[vertex,label=right:{$v_1$}] (vv1_2) at (cvv1_2) {};
\node[vertex,label=left:{$u$}] (hu_2) at (chu_2) {} edge[very thick] (vv1_2);
\node[vertex,label=right:{$v_2$}] (vv2_2) at (cvv2_2) {} edge[very thick] (vv1_2) edge[very thick] (hu_2);
\node[vertex,label=right:{$v_{k_{2}-1}$}] (vvk11_2) at (cvvk11_2) {} edge[very thick,dashed] (vv2_2) edge[very thick] (hu_2);
\node[vertex,label=right:{$v_{k_{2}}$}] (vvk1_2) at (cvvk1_2) {} edge[very thick] (vvk11_2) edge[very thick] (hu_2);

\node at (+8,\stepthree) {$\mathcal{B}'$};

\coordinate (cu1_3) at (0,\stepthree+1);
\coordinate (cu2_3) at (0,\stepthree-1);
\coordinate (ca1_3) at (1,\stepthree);
\coordinate (ca2_3) at (2,\stepthree);
\coordinate (cal2_3) at (5,\stepthree);
\coordinate (cal1_3) at (6,\stepthree);
\coordinate (cx2_3) at (7,\stepthree);
\coordinate (cv2_3) at (7,\stepthree+1);
\coordinate (cvk11_3) at (7,\stepthree+2);
\coordinate (cvk1_3) at (7,\stepthree+3);
\coordinate (cvv2_3) at (7,\stepthree-1);
\coordinate (cvvk11_3) at (7,\stepthree-2);
\coordinate (cvvk1_3) at (7,\stepthree-3);

\fill[blue!25!white] (cu2_3) -- (cu1_3) -- (cvk1_3) -- (cvvk1_3) -- cycle;

\node[vertex,label=left:{$u_1$}] (u1_3) at (cu1_3) {};
\node[vertex,label=left:{$u_2$}] (u2_3) at (cu2_3) {} edge[very thick] (u1_3);
\node[vertex] (a1_3) at (ca1_3) {} edge[very thick] (u1_3) edge[very thick] (u2_3);
\node[vertex] (a2_3) at (ca2_3) {} edge[very thick] (a1_3) edge[very thick] (u1_3) edge[very thick] (u2_3);
\node[vertex] (al2_3) at (cal2_3) {} edge[very thick,dashed] (a2_3) edge[very thick] (u1_3) edge[very thick] (u2_3);
\node[vertex] (al1_3) at (cal1_3) {} edge[very thick] (al2_3) edge[very thick] (u1_3) edge[very thick] (u2_3);
\node[vertex,label=right:{$x_2$}] (x2_3) at (cx2_3) {} edge[very thick] (al1_3) edge[very thick] (u1_3) edge[very thick] (u2_3);

\node[vertex,label=right:{$v_2$}] (v2_3) at (cv2_3) {} edge[very thick] (x2_3) edge[very thick] (u1_3);
\node[vertex,label=right:{$v_{k_{1}-1}$}] (vk11_3) at (cvk11_3) {} edge[very thick,dashed] (v2_3) edge[very thick] (u1_3);
\node[vertex,label=right:{$x'_1$}] (vk1_3) at (cvk1_3) {} edge[very thick] (vk11_3) edge[very thick] (u1_3);

\node[vertex,label=right:{$v_2$}] (vv2_3) at (cvv2_3) {} edge[very thick] (x2_3) edge[very thick] (u2_3);
\node[vertex,label=right:{$v_{k_{2}-1}$}] (vvk11_3) at (cvvk11_3) {} edge[very thick,dashed] (vv2_3) edge[very thick] (u2_3);
\node[vertex,label=right:{$x'_2$}] (vvk1_3) at (cvvk1_3) {} edge[very thick] (vvk11_3) edge[very thick] (u2_3);

\node (F1p) at ($(cu1_3)!0.5!(cvk11_3)!0.33!(cvk1_3)$) {$F'_1$};
\node (F2p) at ($(cu2_3)!0.5!(cvvk11_3)!0.33!(cvvk1_3)$) {$F'_2$};

\end{scope}

\end{tikzpicture}}
\end{center}
 \caption{A~$2$-dimensional illustration of the example constructed in Theorem~\ref{thm:exampleBall}}
 \label{fig:exampleBall}
\end{figure}

Figure~\ref{fig:exampleBall} illustrates the kind of objects that are produced by the previous proof.

%%%%%%% A Hirsch polytope with exponential monotone conservative paths %%%%%%%%%
%%%%%%%%%%%%%%%%%%%%%%%%%%%%%%%%%%%%%%%%%%%%%%%%%%%%%%%%%%%%%%%%%%%%%%%%%%%%%%%%
\subsection{A Hirsch polytope with exponential monotone conservative paths}

We now show a second construction which achieves essentially the same long monotone and conservative paths but in a simplicial complex that is a \emph{polytopal sphere}. The construction is similar in spirit to the previous one, based on the use of one-point-suspensions, but a bit more complicated since in a sphere we cannot have the vertices contained in a unique facet that were instrumental in the previous proof. 

In the inductive step, instead of gluing \emph{stacks} to the one-point-suspension of the previously constructed sphere, we do stellar subdivisions at edges. Let us first highlight some properties of these stellar subdivisions that will be used in the construction:

Let $ab$ be an edge in a simplicial complex $C$, and let $C'$ be the complex obtained by a stellar subdivision of the edge $ab$. Let $v$ be the new vertex.
\begin{itemize}
\item Let $c$ be a vertex of $C$ other than $a$ and $b$ and assume that $abc$ is not a triangle. Then, $\lk_{C'}(v)=\lk_{C}(v)$.
\item The same is true for the links at $a$ (resp.~$b$) except the role of $b$ (resp.~$a$) in it is now played by $c$.
\item The degree of $v$ in the 1-skeleton of $C'$ is $2$ plus the number of triangles containing $ab$ in~$C$, and is bounded above by $\min\{\deg_C(a), \deg_C(b)\}$.
\end{itemize}

In particular, the proof makes use of the following idea. Let $a$ be a vertex in a simplicial complex~$C$ and let $C'$ be obtained from $C$ by stellar subdivision of all the edges $ab$ containing $a$ (in any order). Then, for every vertex $v\ne a$ of $C$ we have that
\[
\vdist_{C'}(a,v) = \vdist_{C}(a,v) + 1.
\]

We are now ready to state and prove the lemma that gives the induction step:

\begin{lemma}
\label{lemma:exampleSphere}
Let $C$ be a $(d-1)$-complex on $N$ vertices. If $C$ has two facets $F_1$ and $F_2$ and vertices $x_i\in F_i$ ($i=1,2$) such that:
\begin{enumerate}
\item Every monotone conservative path from $F_i$ towards $\{x_j\}$ (and ending in a facet containing $\{x_j\}$) ends precisely in $F_j$ and has length at least $L$, for $(i,j)\in\{(1,2),(2,1)\}$.
\item Every monotone conservative path from $F_i$ towards $F_j$ (and ending in $F_j$) has length at least $L$, for $(i,j)\in\{(1,2),(2,1)\}$.
\item $\vdist_C(x_1,x_2)\ge 3$.
\item The degrees of $x_1$ and $x_2$ in (the 1-skeleton of) $C$ are $2d-2$.
\end{enumerate}

Then, there exists a $d$-complex $C'$ on $N':= N+ 4d+1$ vertices having the same properties (1) to (4) with $L':=2L$ and $d'=d+1$, and with respect to facets $F'_i$ and vertices $x'_i\in F'_i$ ($i=1,2$).
\end{lemma}

\begin{proof}
Consider first the complex $\ops_C(x_1)$ and let $u$ and $v$ be the suspension vertices. In it do the following stellar subdivisions:
\begin{itemize}
\item Let $\{y_1,\dots,y_{2d-2}\}$ be the neighbor vertices of $x_2$ in $C$. For each of them, subdivide the edges $uy_i$ and $vy_i$. Call $u_i$ and $v_i$ the new vertices introduced.
\item Then subdivide the edges $ux_2$ and $vx_2$. Call $u'$ and $v'$ the new vertices.
\item Finally subdivide the edges $uu'$ and  $vv'$ and call the new vertices $x'_1$ and $x'_2$. 
\end{itemize}
Let $C'$ be the final complex so obtained. Observe that the number of vertices has increased by one via the one-point-suspension, then by $4d-4$ (vertices $u_i$ and $v_i$, $i=1,\dots, 2d-2$), then, finally by another four ($u'$, $v'$, $x'_1$ and $x'_2$). In particular, the new complex indeed has $N':=N+4d+1$ vertices.

Observe that the links of $u$ and $v$ do not change by the subdivisions, except some of the new vertices introduced play the role of vertices of $C$. 
For example, in $\lk_{C'}(u)$ the role of $x_1$ and $x_2$ is played by $x'_1$ and $v$, respectively, and the role of each neighbor $y$ of $x_1$ is played by the new vertex produced by subdividing $uy$. 
Here we are using the assumption $\vdist(x_1,x_2)\ge 3$. If this didn't hold, some of the $y_i$'s would form a triangle with $u$ and $v$, in which case the subdivision of $uy_i$ would change the link at $v$, and vice-versa.

We let $F'_1$ consist of $u$ together with the vertices of $\lk_{C'}(u)$ corresponding to $F_2$ and let $F'_2$ consist of $v$ together with the vertices of $\lk_{C'}(v)$ corresponding to $F_2$. Put differently:
\[
F'_1:=\{u,x'_1\}\cup \{y'_i: y_i \in F_2\},
\qquad
F'_2:=\{v,x'_2\}\cup \{y'_i: y_i \in F_2\}.
\]
Let us show some properties of this construction, among which are the claimed properties (1) to (4).

\begin{itemize}
\item Let us look at the degree of $x'_1$ (and~$x_2$ symmetrically). Observe that the number $2d-2$ of neighbors of $x_2$ in $C$ equals the number of triangles containing $ux_2$ in $\ops_C(x_1)$. This number does not change by the subdivision of the edges $uy_i$, and it becomes the number of triangles containing $uu'$ when we subdivide $ux_2$. So, when we create the vertex $x'_1$ by subdividing $uu'$, the degree of the new vertex $x'_1$ equals 2 plus that number, which shows property (4).

\item The path $(x'_1,u,v,x'_2)$ is the unique shortest path from $x'_1$ to $x'_2$. Indeed, observe that any other path from $x'_1$ to $x'_2$ (or from a vertex on the ``$u$ side'' to a vertex on the ``$v$ side'' of the construction, for that matter) needs to pass through a vertex of $C$. The claim then follows by noticing that no vertex of $C$ is a neighbor of neither $x'_1$ nor $x'_2$, so all other paths between them have length at least four. This implies part (3), but it also shows that:
\begin{itemize}
\item $(u,v)$ is the unique shortest vertex-path between $F'_1$ and $F'_2$.
\item $(u,v,x'_2)$ is the unique shortest vertex-path between $F'_1$ and $x'_2$.
\item $(x'_1,u,v)$ is the unique shortest vertex-path between $x'_1$ and $F'_2$.
\end{itemize}

\item Then every monotone conservative path from $F'_1$ to $x'_2$ has $(u,v,x'_2)$ as its sequence of anchors. While we are in the first anchor $u$ the path is a monotone conservative path along $\lk_{C'}(u)$ towards the vertex set $S':=\{v\}$, which is the same as a monotone conservative path from $F_2$ to $x_1$ in $C$. By hypothesis (1), this path has length at least $L$ and finishes in the facet $\ops_C(F_1)$. The part anchored at $v$ is, by the same arguments, a monotone conservative path from $\ops_C(F_1)$ to $x'_2$, which finishes in $F'_2$ and has length also at least~$L$ by the same hypothesis. This proves (1) for $C'$.

\item Similarly, every monotone conservative path from $F'_1$ towards $F'_2$ has $(u,v)$ as its first two anchors. While we are anchored at~$u$ the path is a monotone conservative path along $\lk_{C'}(u)$ towards the vertex set $S':=\{v\}$, which is the same as a monotone conservative path from $F_2$ to $x_1$ in $C$. By hypothesis (1), this path has length at least $L$ and finishes in the facet $\ops_C(F_1)$. The rest of the path (the part anchored at $v$, and what comes later, since $v$ belongs to our target set $S$)  is a monotone conservative path from $\ops_C(F_1)$ to $F'_2$ and has length also at least $L$ by hypothesis (2). This proves (2) for $C'$.\qedhere
\end{itemize}
\end{proof}

\begin{theorem}
\label{thm:exampleSphere}
For every $d\ge 2$ and every $N\ge 4$, there is a vertex-decomposable polytopal $(d-1)$-sphere $\sphere_d$ with $N+\Theta(d^2)$ vertices and two facets $F_1$ and $F_2$ in it such that every  monotone conservative path between them has length at least $2^{d-3} N$.
\end{theorem}

See Figure~\ref{fig:vertexDecomposablePolytope} for a~$2$-sphere example.

\begin{proof}
Let $\sphere_2$ be a cycle with $N$ vertices and iterate on it the process of Lemma~\ref{lemma:exampleSphere}.
Observe that the construction in the lemma preserves polytopality and vertex-decomposability, since both one-point-suspension and stellar subdivision preserve them.
For vertex-decomposability this follows from Lemma~\ref{lem:ops-vertex-decomposable}.
For polytopality, it suffices to add a vertex outside the polytope close enough to the barycenter of the face which is stellar-subdivided to realize the stellar subdivision.
\end{proof}

\begin{figure}
 \centerline{\input{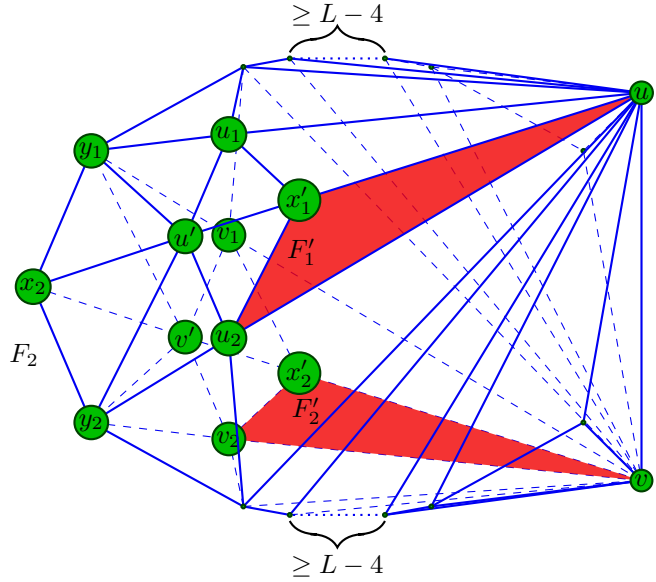}}
 \caption{A~$2$-sphere illustration of the example constructed in Theorem~\ref{thm:exampleSphere}}
 \label{fig:vertexDecomposablePolytope}
\end{figure}

\begin{remark}
 \begin{itemize}
  \item Observe that, for fixed $d$, the bounds of Theorems~\ref{thm:exampleBall} and~\ref{thm:exampleSphere} are both within a $1+\varepsilon$ ratio of the upper bounds of Theorem~\ref{thm:BarnetteLarman}, with $\varepsilon$ going to zero as $N$ goes to infinity.  
  \item The fact that the polytopes constructed in Theorem~\ref{thm:exampleSphere} are vertex-decomposable implies that they satisfy the Hirsch bound~\cite{ProvanBillera}, and it also dramatically illustrates the need of flagness for the proof of Adiprasito and Benedetti~\cite{AdiprasitoBenedetti} to work.
 \end{itemize}
\end{remark}

\bibliographystyle{alpha}
\bibliography{bibliography}

\end{document}